\documentclass[12pt]{amsart}
\usepackage{amssymb}
\usepackage{tikz-cd,adjustbox}
\usepackage[all]{xy}

\usepackage[colorlinks=true,linkcolor=magenta,citecolor=magenta]{hyperref}

\newtheorem{theorem}{Theorem}[section]
\newtheorem{definition}[theorem]{Definition}
\newtheorem{proposition}[theorem]{Proposition}
\begin{document}

\title[The Frucht property]{The Frucht property in the quantum group setting}

\author{T. Banica}
\address{T.B.: Department of Mathematics, University of Cergy-Pontoise, Cergy-Pontoise, France. {\tt teo.banica@gmail.com}}

\author{J.P. McCarthy}
\address{J.P.M.: Department of Mathematics, Munster Technological University, Cork, Ireland. {\tt jeremiah.mccarthy@cit.ie}}

\subjclass[2010]{46L65 (46L53, 81R50)}
\keywords{Quantum permutation, Quantum automorphism}

\begin{abstract}
A classical theorem of Frucht states that any finite group  appears as the automorphism group of a finite graph. In the quantum setting the problem is to understand the structure of the compact quantum groups  which can appear as quantum automorphism groups of finite graphs. We discuss here this question, notably with a number of negative results.
\end{abstract}
\setcounter{tocdepth}{1}
\maketitle

\baselineskip=16.5pt
\tableofcontents
\baselineskip=14pt

\section*{Introduction}
A natural question in group theory is whether or not a  group can be realised as the symmetry group of some object. There are various positive answers to this question, but perhaps one of the easiest to parse is a 1939 theorem of Frucht who showed that every finite group is the automorphism group of a finite graph \cite{fru}, a result that was subsequently extended to the infinite setting independently by de Groot \cite{deg}, and Sabidussi \cite{sab}.

\bigskip

About half a century on from Frucht's theorem, the Woronowicz school of compact quantum groups was established with three seminal papers: first \cite{wo1} and later \cite{wo2,wo3}. It is difficult to succinctly chart the subsequent investigation of these  quantum groups, but the influential examples, including the quantum permutation group $S_N^+$, constructed by Wang \cite{wa1,wa2} in the 1990s can be cited as providing much impetus for an intensive study of these and other examples  by various authors (including the first author here) in the 2000s, into the 2010s, and indeed to the present day.

\bigskip

One of the most natural urges in this programme is to direct to compact quantum groups questions asked and answered in classical group theory. For example, does a compact quantum group comprise the quantum symmetry group of some object?

\bigskip

A finite graph $X$ with $N$ vertices has an automorphism group $G(X)\subset S_N$, as well as a quantum automorphism group $G^+(X)\subset S_N^+$. We have $G(X)\subset G^+(X)$, and if this inclusion is an equality, we say that $X$ has no quantum symmetry. The contrary can happen as well, and in this case the study of $G^+(X)$ is an interesting question. All this goes back to old work from the mid 2000s, mainly by: the first author, Bichon, and their collaborators \cite{ba3,ba4,bb1,bb2,bbg,bbc,bi1}.

\bigskip

The subject was reinvigorated in the  mid to late 2010s, and in fact recovered its full front-line status from the mid 2000s, on one hand due to solutions found for some old questions, and on the other hand, due to a number of new motivations, coming as a complement to some of the original motivations.

\bigskip

 Lupini, Man\v cinska and Roberson did in their ground-breaking 2017 paper \cite{lmr} a number of unexpected things, namely: a key  extension of Bichon's work on orbits \cite{bi3} to orbitals; a counterexample to an old conjecture from \cite{ba4}, stating that quantum transitivity implies transitivity; proved that the probability that a random graph on $N$ vertices has quantum symmetry goes to zero as $N\rightarrow \infty$;  and, most importantly, found a connection with quantum information theory and nonlocal games. Another important paper, with some partly related results by Musto, Reutter, Verdon \cite{mrv}, appeared at the same time.

\bigskip

 Also at the same time, Schmidt managed to restart the computation work of $G^+(X)$ for graphs with a small number of vertices, solving an old question going back to \cite{bb2} regarding the Petersen graph \cite{sc1}. Later on, with \cite{sc2,sc3} and other papers, he managed to turn the whole program into something that can be labelled as ``high level discrete mathematics'', which is interesting and fruitful for both quantum permutations, and for advanced graph theory. Also a number of interesting advances on a number of old open questions were made in the PhD work of Chassaniol, done under the guidance of Bichon, published later in \cite{ch1,ch2}.

\bigskip

Finally there were big collaborations \cite{bce,els}, mixing operator algebra and quantum group people, quantum physicists, pure algebraists, graph theorists, and many more.

\bigskip

The present paper is somehow dual to previous studies of quantum automorphism groups, which start with a finite graph and study:
$$X\longrightarrow G^+(X).$$
Instead we will start with a compact quantum group $G$, and study all the (isomorphism classes of) finite graphs it acts on:
$$G\longrightarrow \{X_a\}.$$
Our big question, naturally leading on from Frucht's theorem, is whether any of these finite graphs has quantum automorphism group $G^+(X_a)=G$?

\bigskip

Our study leans heavily both on the orbit theory of Bichon \cite{bi3}, as well as ideas, notably orbital theory, from the paper of Lupini, Man\v cinska and Roberson \cite{lmr}. Key to the whole business is connecting an old observation of Bichon \cite{bi2} with orbital theory.

\bigskip

The paper is organised as follows: 1-3 are preliminary sections, on the quantum permutation groups, and quantum automorphism groups of finite graphs, in 4 we present some quantum permutation groups of possible interest in this study, and in 5-8 we discuss the Frucht property, with a number of negative results.

\section{Quantum permutation groups}
We use the compact matrix quantum group formalism of Woronowicz \cite{wo1,wo2}, under the supplementary assumption $S^2=\operatorname{id}$. The axioms are as follows:

\begin{definition}\label{DEF1.1}
A Woronowicz algebra is a $\mathrm{C}^*$-algebra $A$, given with a fundamental unitary representation $u\in M_N(A)$ whose entries generate $A$, such that there exist morphisms
$$\Delta:A\to A\otimes A,$$
$$\varepsilon:A\to\mathbb C,$$
$$S:A\to A^{opp},$$
called comultiplication, counit, and antipode, given by the formulae
$$\Delta(u_{ij})=\sum_ku_{ik}\otimes u_{kj},$$
$$\varepsilon(u_{ij})=\delta_{ij},$$
$$S(u_{ij})=u_{ji}^*,$$
on the standard generators $u_{ij}$.
\end{definition}

Here $\otimes$ is the minimal tensor product of $\mathrm{C}^*$-algebras, and $A^{opp}$ is the opposite $\mathrm{C}^*$-algebra, with multiplication $a\cdot b=ba$. The basic examples are as follows:

\bigskip

(1) Given a compact Lie group, $G\subset U_N$, the algebra $A=C(G)$ of continuous functions $f:G\to\mathbb C$, together with the matrix formed by the standard coordinates, $u_{ij}:g\to g_{ij}$, is a Woronowicz algebra. This algebra is commutative, and the Gelfand theorem shows that any commutative Woronowicz algebra is of this form. See \cite{wo1}.

\bigskip

(2) Given a finitely generated discrete group $\Gamma=\langle g_1,\ldots,g_N\rangle $, the full group $\mathrm{C}^*$-algebra $A=\mathrm{C}^*(\Gamma)$ is a Woronowicz algebra too, with the diagonal matrix formed by the generators, $u=\operatorname{diag}(g_1,\ldots,g_N)$. This algebra $A=\mathrm{C}^*(\Gamma)$ is cocommutative, in the sense that $\Sigma\Delta=\Delta$, where $\Sigma(a\otimes b)=b\otimes a$ is the flip. Conversely, any cocommutative Woronowicz algebra can be shown to be of this form, up to identifying the different group $\mathrm{C}^*$-algebras that $\Gamma$ might have, in the non-amenable case. See \cite{wo1}.

\bigskip

(3) The most basic examples appear by intersecting (1,2). Indeed, the  compact abelian Lie groups $G\subset U_N$ are in Pontrjagin duality, $G=\widehat{\Gamma}$ and $\Gamma=\widehat{G}$, with the finitely generated discrete abelian groups $\Gamma=\langle g_1,\ldots,g_N\rangle $, and in this situation, the Fourier transform gives an identification of Woronowicz algebras $C(G)=\mathrm{C}^*(\Gamma)$. See \cite{wo1}.

\bigskip

Based on these examples and observations, let us formulate:

\begin{definition}\label{DEF1.2}
Given a Woronowicz algebra $(A,u)$, we write
$$A=C(G)=\mathrm{C}^*(\Gamma),$$
and call $G,\Gamma$ compact and discrete quantum groups. Also, we write
$$G=\widehat{\Gamma}\quad,\quad \Gamma=\widehat{G},$$
and call $G,\Gamma$ dual to each other. We say that $G$ is a finite quantum group if $C(G)$ is finite dimensional.  Finally, we identify two Woronowicz algebras
$$(A,u)=(B,v)$$
when the  $*$-algebras $\mathcal A=\langle u_{ij}\rangle $ and $\mathcal B=\langle v_{ij}\rangle $ are isomorphic, via $u_{ij}\to v_{ij}$.
\end{definition}

\bigskip

The basic theory of the compact and discrete quantum groups, developed by Woronowicz in \cite{wo1,wo2}, consists in an existence result for the Haar measure, an extension of the classical Peter--Weyl theory, and an extension of Tannakian duality as well.

\bigskip

\begin{definition}\label{DEFFix}
Let $(C(G),u)$ and $(C(G),v)$ be Woronowicz algebras with the same underlying $\mathrm{C}^*$-algebra $C(G)$, but different fundamental unitary matrices, $u\in M_{N_1}(C(G))$ and $v\in M_{N_2}(C(G))$. Where
$$\Delta_1(u_{ij})=\sum_{k=1}^{N_1}u_{ik}\otimes u_{kj}\text{ \quad and\quad }\Delta_2(v_{ij})=\sum_{k=1}^{N_2} v_{ik}\otimes v_{kj},$$
if $\Delta_1=\Delta_2$, we consider $(C(G),u)$ and $(C(G),v)$ as representations of the same ``abstract'' compact quantum group $G$.
\end{definition}
This definition is justified by the fact that if $(C(G),u)$ is a Woronowicz algebra, then $C(G)$ and $\Delta:C(G)\rightarrow C(G)\otimes C(G)$ also satisfy the axioms of an ``abstract'' compact quantum group, as in \cite{wo3}. This allows different Woronowicz algebra representations of a compact quantum group.

\bigskip

Following Wang \cite{wa2}, the quantum permutation groups are constructed as follows:

\begin{definition}\label{DEF1.3}
The quantum permutation group $S_N^+$ is the compact quantum group whose Woronowicz algebra is the universal $\mathrm{C}^*$-algebra given by
$$C(S_N^+)=\mathrm{C}^*\left((u_{ij})_{i,j=1\ldots N}\Big|\,u={\rm magic}\right),$$
with ``magic'' meaning formed of projections, in the $\mathrm{C}^*$-algebra sense $(p^2=p=p^*)$, which must sum up to $1$ on each row and each column.
\end{definition}

Here the fact that $C(S_N^+)$ is indeed a Woronowicz algebra is standard, coming from the fact that if $(u_{ij})$ is magic, then so are the following matrices:
$$U_{ij}=\sum_ku_{ik}\otimes u_{kj},$$
$$1_{ij}=\delta_{ij},$$
$$v_{ij}=u_{ji}.$$

Thus we can define indeed morphisms $\Delta,\varepsilon,S$ as in Definition \ref{DEF1.1}, by using the universality property of $C(S_N^+)$. See \cite{wa2}.

\begin{definition}
A quantum permutation group is a compact quantum group that admits a fundamental magic representation. When we write
$$G\subset S_N^+,$$
we assume a fixed fundamental magic representation $u\in M_N(C(G))$, and subsequently $u_{ij}$ referring to a generator of $C(G)$ rather than of $C(S_N^+)$.
\end{definition}

\bigskip

The definitions and terminology of Definition \ref{DEF1.3} are motivated by a series of key observations and results from \cite{wa2}, as follows:

\bigskip

(1) The symmetric group $S_N$, when regarded as the group of permutations of the $N$ coordinate axes of $\mathbb R^N$, becomes a matrix group, $S_N\subset O_N$. In this picture, the standard coordinates $v_{ij}:S_N\to\mathbb R$ are the entries of the permutation matrices, given by:
$$v_{ij}=\chi\left(\sigma\in S_N\Big|\sigma(j)=i\right).$$

(2) These characteristic functions $v_{ij}$ are abstract projections in the $\mathrm{C}^*$-algebra sense, $v_{ij}^2=v_{ij}=v_{ij}^*$, summing to 1 on each row and column of $v=(v_{ij})$. Thus the matrix $v$ is magic, and so we have a quotient map as follows:
$$C(S_N^+)\to C(S_N),$$
$$u_{ij}\to v_{ij}.$$

(3) This quotient map is a morphism of Woronowicz algebras, and so corresponds to an embedding of compact quantum groups, as follows:
$$S_N\subset S_N^+.$$

(4) The classical version, $G_{\text{class}}\subset G$, of a compact quantum group $G$ is the Gelfand spectrum of the abelianisation of $C(G)$. Recall the abelianisation of $C(G)$ is the commutative $\mathrm{C}^*$-algebra $C(G)/I$, with $I\subset C(G)$ being the commutator ideal. In turn $G$ is said to be a ``liberation'' of $G_{\text{class}}$.  The embedding $S_N\subset S_N^+$ is a ``liberation''. This comes indeed from the following presentation result for the algebra $C(S_N)$, which is standard:
$$C(S_N^+)=\mathrm{C}^*_{comm}\left((v_{ij})_{i,j=1\ldots N}\Big|\,v={\rm magic}\right).$$

(5) Now since we have a liberation $S_N\subset S_N^+$, it makes sense to say that $S_N^+$ is a ``quantum permutation group''. Moreover, a quick comparison between the formula in Definition \ref{DEF1.3} for $C(S_N^+)$ and the one above for $C(S_N)$ suggests that $S_N^+$ should be the ``biggest'' such liberation, and so is ``the'' quantum permutation group.

\bigskip

(6) These latter considerations are confirmed by a precise result, stating that $S_N^+$ is the biggest compact quantum group acting on $X=\{1,\ldots,N\}$. To be more precise, $C(S_N^+)$ is the biggest Woronowicz algebra coacting on $C(X)$, in the following sense:
$$\Phi:C(X)\to C(X)\otimes C(G),$$
$$e_j\to\sum e_i\otimes u_{ij}.$$

(7) The proof of this latter fact is based on the observation, coming from some elementary computations, that $\Phi$ defined as above is a coaction precisely when $u=(u_{ij})$ is magic. Technically speaking, one must assume in all this that the action $G\curvearrowright X$ preserves the counting measure on $X$, with this being not automatic in the quantum case.

\bigskip

As an extremely important question now, is this abstract object $S_N^+$ that we just introduced really bigger than $S_N$? And also, if bigger, is it really non-finite, needing the ``compact'' formalism? The answers here, which add to the above, and fully justify Definition \ref{DEF1.3}, are as follows, also from \cite{wa2}:

\begin{theorem}
The embedding $S_N\subset S_N^+$ is an isomorphism at $N=1,2,3$, but not at $N\geq4$, where $S_N^+$ is a non-classical, non-finite compact quantum group.
\end{theorem}

\begin{proof}
This can be proved by a case-by-case analysis, as follows:

\bigskip

(1) Trivial.

\bigskip

(2) The $N=2$ statement is clear, because a $2\times2$ magic matrix must be as follows, with $p$ being a projection:
$$u=\begin{pmatrix}p&1-p\\ 1-p&p\end{pmatrix}.$$

Now since the entries of $u$ commute, it follows that $C(S_2^+)$ is commutative, and so $S_2^+$ collapses to its classical version, which is $S_2$.

\bigskip

(3) At $N=3$ now, by using the same idea as in the $N=2$ case, we must prove that the entries of any $3\times3$ magic matrix commute. This is something quite tricky, and there are 4 known proofs here \cite{ba3}, \cite{lmr}, \cite{mcc}, \cite{wa1}. According to the proof in \cite{mcc}, which is the most recent, it suffices to show that $u_{11}u_{22}=u_{22}u_{11}$, by showing:
$$u_{11}u_{22}=u_{11}u_{33}=u_{22}u_{33}=u_{22}u_{11}.$$

But these three equalities come respectively from:
$$u_{11}(u_{21}+u_{22}+{u_{23}})=u_{11}(u_{13}+u_{23}+u_{33})$$
$$(u_{11}+u_{21}+u_{31})u_{33}=(u_{21}+u_{22}+u_{23})u_{33}$$
$$u_{22}(u_{31}+u_{32}+u_{33})=u_{22}(u_{11}+u_{21}+{u_{31}})$$

(4) At $N=4$ we must prove that the algebra $C(S_4^+)$ is noncommutative and infinite dimensional, and this comes from the fact that the following matrix is magic, for any choice of two projections $p,q$ on a Hilbert space $H$:
$$u=\begin{pmatrix}
p&1-p&0&0\\
1-p&p&0&0\\
0&0&q&1-q\\
0&0&1-q&q
\end{pmatrix}.$$

Thus, if we choose $p,q$ as the generators of $\mathrm{C}^*(p,q)$, the universal unital $\mathrm{C}^*$-algebra generated by two projections,  we obtain the result, due to the quotient map $C(S_4^+)\to\mathrm{C}^*(p,q)$. It is worth noting that where $D_\infty=\langle a,b \,|\,a^2=b^2=1\rangle$ is the infinite dihedral group, $\mathrm{C}^*(p,q)\cong C(\widehat{D_\infty})$ via the isomorphism $p\to (e+a)/2$ and $q\to (e+b)/2$.

\bigskip

(5) For $N\geq 5$, let $v$ be the magic unitary of $C(S_4^+)$, and $1_{N-4}\in M_{N-4}(C(S_4^+))$ the diagonal matrix with the unit of $C(S_4^+)$ along the diagonal. The magic unitary $u=\operatorname{diag}(v,1_{N-4})$ shows that we have an embedding $S_4^+\subset S_N^+$.
\end{proof}

\section{Orbits, group duals, and orbitals}\label{SEC2}

We have the following notion, due to Bichon \cite{bi3}:

\begin{definition}\label{DEF2.1}
Given a quantum permutation group $G\subset S_N^+$, we define its orbits as being the equivalence classes under the following relation:
$$i\sim j\iff u_{ij}\neq0.$$
We say that $G$ is transitive when we have only one orbit (i.e. $u_{ij}\neq0$ for any $i,j$).
\end{definition}

Observe that in the classical case, $G\subset S_N$, we obtain indeed the classical notation of an orbit, because here the standard coordinates $u_{ij}$ are the following characteristic functions:
$$u_{ij}=\chi\left(\sigma\in G\Big|\sigma(j)=i\right).$$

In general, the fact that $\sim$ defined as above is transitive, reflexive and symmetric comes by examining the formulae of $\Delta,\varepsilon,S$ on the standard generators $u_{ij}$. We refer to \cite{bi3} for full details here. Bichon introduced the orbits mainly in order to solve a key problem, namely the classification of the group dual embeddings $\widehat{\Gamma}\subset S_N^+$. In one sense, we have the following construction, which is elementary:

\begin{proposition}
Given a finitely generated group $\Gamma=\langle g_1,\ldots,g_k\rangle $, with the generators having finite order, $g_i^{N_i}=1$, we have a quantum group embedding
$$\widehat{\Gamma}\subset S_N^+,$$
with $N=N_1+\ldots+N_k$, as follows, using the Fourier identifications $\widehat{\mathbb Z}_{N_p}\simeq\mathbb Z_{N_p}$,
\begin{eqnarray*}
\widehat{\Gamma}
&\subset&\widehat{\mathbb Z_{N_1}*\ldots*\mathbb Z_{N_k}}=\widehat{\mathbb Z}_{N_1}\,\hat{*}\,\ldots\,\hat{*}\,\widehat{\mathbb Z}_{N_k}\\
&\cong&\mathbb Z_{N_1}\,\hat{*}\,\ldots\,\hat{*}\,\mathbb Z_{N_k}\subset S_{N_1}^+\,\hat{*}\,\ldots\,\hat{*}\,S_{N_k}^+\\
&\subset&S_N^+,
\end{eqnarray*}
and with $\hat{*}$ being the dual free product operation for compact quantum groups.
\end{proposition}

\begin{proof}
Our assumption $\Gamma=\langle g_1,\ldots,g_k\rangle $, with the generators satisfying $g_p^{N_p}=1$, means that $\Gamma$ appears as a quotient of a free product of cyclic groups, as follows:
$$\mathbb Z_{N_1}*\ldots*\mathbb Z_{N_k}\to\Gamma.$$

Thus the dual $\widehat{\Gamma}$ appears as a closed subgroup of the dual $\widehat{\mathbb Z_{N_1}*\ldots*\mathbb Z_{N_k}}$, and the continuation is standard, by using Wang's operation $\hat{*}$ from \cite{wa1}, defined by:
$$C(G\,\hat{*}\,H)=C(G)*C(H).$$

Let us record as well the explicit formula of the embedding. For each $g_p$ the magic unitary $u^{g_p}=(u^{g_p}_{kl})$ of the corresponding dual $\widehat{\mathbb Z}_{N_p}$ is as follows, with $w=e^{2\pi i/N_p}$:
\begin{equation}u^{g_p}_{kl}=\frac{1}{N_p}\sum_{m=1}^{N_p}w^{(k-l)m}g_p^m.\label{fouriermagi}\end{equation}

A magic unitary for $\widehat{\Gamma}$ is then given by the following formula:
$$u=\begin{pmatrix}
u^{g_1}\\
&\ddots\\
&&u^{g_k}
\end{pmatrix}.$$

Finally, let us record the following useful inversion formula:
$$g_p=\sum_{k=1}^{N_p}w^{1-k}u_{k1}^{g_p}.$$

\end{proof}

We have the following converse of the above result, from \cite{bi3}:

\begin{theorem}
Any group dual subgroup $\widehat{\Gamma}\subset S_N^+$ appears as above, from a quotient
$$\mathbb Z_{N_1}*\ldots*\mathbb Z_{N_k}\to\Gamma$$
with the numbers $N_1,\ldots,N_k$ summing up to $N$.
\end{theorem}

\begin{proof}
Assume that we have a group dual subgroup $\widehat{\Gamma}\subset S_N^+$. The corresponding magic unitary must be of the following form, with $h_p\in\Gamma$ and $U\in U_N$:
$$u=U
\begin{pmatrix}
h_1\\
&\ddots\\
&&h_N
\end{pmatrix}
U^*.$$

Consider now the orbit decomposition for $\widehat{\Gamma}\subset S_N^+$, coming from Definition \ref{DEF2.1}:
$$N=N_1+\ldots+N_k.$$

We conclude that $u$ has a $N=N_1+\ldots+N_k$ block-diagonal pattern, and so that $U$ has as well this $N=N_1+\ldots+N_k$ block-diagonal pattern. But this discussion reduces our problem to its $k=1$ particular case, with the statement here being that the cyclic group $\mathbb Z_N$ is the only transitive group dual $\widehat{\Gamma}\subset S_N^+$. The proof of this latter fact being elementary, we obtain the result. For full details here, we refer to \cite{bi3}.
\end{proof}

In fact,  the orbit theory of Bichon tells us something about all embeddings $G\subset S_N^+$.
\begin{definition}
Given a quantum permutation group $G$, a transitive magic representation of $C(G)$ is a magic unitary $u^{p}\in M_{N_p}(C(G))$ such that:
\begin{enumerate}
  \item $u_{ij}^{p}\neq 0$,
  \item $\displaystyle \Delta(u_{ij}^{p})=\sum_{k=1}^{N_p}u_{ik}^{p}\otimes u_{kj}^{p}$,
  \item $\varepsilon(u_{ij}^{p})=\delta_{ij}$,
  \item $S(u^{p}_{ij})=u_{ji}^{p}$.
\end{enumerate}
Say that a transitive magic representation $u^{p}$ is equivalent to a transitive magic representation $u^{q}$ if there exists a permutation matrix $P$ such that $u^{q}=Pu^{p}P^{-1}$.
\end{definition}

\bigskip

Let $M=\{u^p\}_{p\in M}$ be the set of equivalence classes of transitive magic representations of $C(G)$.
\begin{proposition}\label{PROPBLO}
Let $G\subset S_N^+$. There exists a permutation matrix $P$ such that, where $v=PuP^{-1}$:
\begin{enumerate}
  \item $v$ is block diagonal, with  transitive magic representation blocks,
  \item the transitive magic representations that appear in $u$ are a finite subset $A\subset M$, and can appear with multiplicity,
  \item the entries of the blocks generate $C(G)$.
\end{enumerate}
\end{proposition}
\begin{proof}
The orbit relation is an equivalence relation and so partitions:
$$\{1,\dots,N\}=\bigsqcup_{p=1}^k X_p.$$
Where $|X_p|=N_p$, let $P$ be a permutation of $\{1,\dots,N\}$ that maps
$$X_1\rightarrow \{1,\dots,N_1\},\,X_2\rightarrow \{N_1+1,\dots,N_1+N_2\},\, \text{etc.}.$$
Then $v=P^{-1}uP$ is a block diagonal fundamental magic representation. The rest follows immediately.
\end{proof}

\bigskip

This gives a theoretical but quickly impractical way of classifying all embeddings $G\subset S_N^+$ up to equivalence.
\begin{enumerate}
  \item Find all the  subsets $A\subset M$ of size $|A|\leq N$.
  \item From these find the subsets $A$ such that
  $$\langle u^p_{ij}\colon p\in A,\,i,j=1,\dots,N_p\rangle= C(G).$$
  \item For $m_p\in \mathbb{N}_{+}$, solve:
  $$\sum_{p\in A}m_pN_p=N.$$
  \item Each solution $(m_p)_{p\in A}\in \mathbb{N}_+^{|A|}$ gives, where each $u^{p_l}$ appears with multiplicity $m_{p_l}$, an embedding:
  $$u=\operatorname{diag}(u^{p_1},\dots,u^{p_1},u^{p_2},\dots,u^{p_2},\dots,u^{p_{|A|}},\dots,u^{p_{|A|}}),$$
  and, up to equivalence, all embeddings $G\subset S_N^+$ are of this form.
\end{enumerate}
This is doable for small finite quantum groups but  quickly becomes difficult.

\bigskip

Beyond finding the set of equivalence classes  $M$ of transitive magic representations, this is slightly oblique to the Frucht problematics of sections \ref{SECDuals}-\ref{SEC8}, where we will have to consider arbitrary embeddings.

\bigskip

As was observed in \cite{lmr}, the orbit theory of Bichon also has something to say regarding the Haar measure:
\begin{proposition}\label{THORE}
If $G\subset S_N^+$, then where $N_p$ is the size of the orbit of $j$:
$$\int_G u_{ij}=\begin{cases}
                  1/N_p, & \mbox{if } i\sim j, \\
                  0, & \mbox{otherwise}.
                \end{cases}$$
\end{proposition}
In relation with above, we have as well the notion of orbitals, also from \cite{lmr}:

\begin{definition}
Given a quantum permutation group $G\subset S_N^+$, we define its orbitals as being the equivalence classes under the following relation:
$$(i,k)\sim(j,l)\iff u_{ij}u_{kl}\neq0.$$
We say that $G$ is doubly transitive when we have only two orbitals, which amounts to saying that we must have $u_{ij}u_{kl}\neq0$ for any $i\neq k$ and $j\neq l$.
\end{definition}

In what regards the higher orbitals, things do not seem to work here, as noticed in \cite{lmr}. As explained in \cite{mcc}, with examples, the 3-orbital
 can fail to be transitive when the diagonal elements of the magic unitary generate a commutative algebra.

\section{Finite graphs and the Frucht property}\label{SEC3}
Let $X=(V,E)$ be a finite graph, with vertex set $V$ of size $N$, and edge set $E$ a symmetric, irreflexive relation on $V$. When $(i,j)\in E$ then $(j,i)\in E$ also, but we draw only a single, undirected edge between vertices $i$ and $j$. That $X$ has no loops is accounted for by the edge relation assumed irreflexive.

\bigskip

 A permutation $\sigma\in S_N$ leaves invariant the edges precisely when it commutes with the adjacency matrix $d\in M_N(0,1)$. Thus, the algebra of functions on the automorphism group $G(X)\subset S_N$ is given by:
$$C(G(X))=C(S_N)/\langle du=ud\rangle.$$

By dropping the assumption that this algebra commutes, we are led to the following construction, going back to \cite{ba3}:

\begin{definition}
Given a graph $X$ having $N$ vertices, the construction
$$C(G^+(X))=C(S_N^+)/\langle du=ud\rangle, $$
where $d\in M_N(0,1)$ is the adjacency matrix of $X$, produces a closed subgroup
$$G^+(X)\subset S_N^+$$
called the quantum automorphism group of $X$.
\end{definition}

Observe that we have an inclusion $G(X)\subset G^+(X)$, which is a liberation, in the sense that the classical version of $G^+(X)$ is the group $G(X)$. This follows indeed from the fact that $S_N\subset S_N^+$ itself is a liberation, explained in section 1 above.

\bigskip

In the sequel we will be interested in quantum subgroups of $G^+(X)$ acting on $X$.
\begin{definition}
Let $X$ be a graph with adjacency matrix $d\in M_N(0,1)$. If $G\subset S_N^+$ such that $du=ud$, we write $G\curvearrowright X$, and have $G\subset G^+(X)$. On the other hand, where $X$ is a graph on $N$ vertices, writing $G\curvearrowright X$ implies a fixed embedding $G\subset S_N^+$ such that $du=ud$.
\end{definition}

\bigskip

As a first observation, quantum automorphisms do exist, for instance because the graph having $N\geq 4$ vertices and no edges has $S_N^+\neq S_N$ as quantum automorphism group. There are many other examples, and all this suggests the following definition:

\begin{definition}
Let $X$ be a finite graph with its classical and quantum automorphism groups $G(X)\subset G^+(X)$. We say that:
\begin{enumerate}
\item $X$ is vertex-transitive when $G(X)$ is transitive.

\item $X$ is quantum vertex-transitive when $G^+(X)$ is transitive.

\item $X$ has no quantum symmetry when $G^+(X)=G(X)$.
\end{enumerate}
\end{definition}

We are interested in determining how the following result from graph theory, established by Frucht in the 1930s, extends to the quantum group setting:

\begin{theorem}[Frucht]\label{FRT}
Every finite group  is the automorphism group of a  finite graph.
\end{theorem}

\begin{proof}
Given a finite group $G$, let $S=\{s_1,\dots,s_m\}$ be a generating set that doesn't contain the identity nor both a non-involution and its inverse. Consider the Cayley graph $X_0=(G,E_0)$ with respect to $S$. The edge set $E_0\subset G\times G\times\{1,\dots,m\}$ consists of directed edges $(s_k^{-1}g_j,g_j,k)$ of colour $k$ between $s_k^{-1}g_{j}$ and $g_j$. It can be shown that the group of colour-preserving automorphisms of $X_0$ is isomorphic to $G$.  Where the vertical path graphs below contain $k$ and $k+1$ vertices respectively, replace each  directed  edge from $g_i$ to $g_j$ of colour $k$ in $X_0$, with a copy of the following asymmetric, undirected graph:
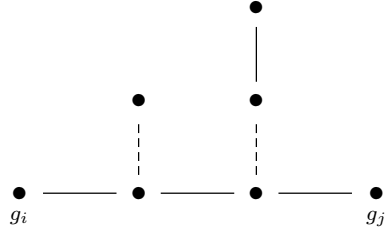
\begin{center}
\begin{tikzcd}
                  &                   & \bullet                  &         \\
                  &  \bullet                 & \bullet  \arrow[u,dash]                   &         \\
\underset{g_i}{\bullet} \arrow[r,dash] & \bullet \arrow[r,dash] \arrow[u,dashrightarrow,no head] & \bullet  \arrow[u,dashrightarrow,no head] \arrow[r,dash] & \underset{g_j}{\bullet}
\end{tikzcd}
\end{center}

It can be shown that the resulting finite graph has automorphism group $G$. See \cite{fru} for full details.
\end{proof}
We note that the choice of graph to replace a directed edge of color $k$ is non-canonical, and, being suitably careful, it is possible to instead use different sets of asymmetric graphs.

\bigskip

The question of whether this result extends to the quantum group setting is a natural one, and was posed explicitly in \cite{mro}, and implicitly in \cite{gro}.

\begin{definition}
A compact quantum group $G$ has the Frucht property if there exists a finite graph $X$ such that $G=G^+(X)$.
\end{definition}
The tables in Schmidt's thesis \cite{sc4} comprise a collection of some of the compact quantum groups known to have the Frucht property. 

\bigskip

We have the following negative result, coming from \cite{bbn}:

\begin{theorem}
The following happen:
\begin{enumerate}
\item Any finite group is a permutation group.

\item There are finite quantum groups which are not quantum permutation groups.

\item Thus there are finite quantum groups which do not have the Frucht property.
\end{enumerate}
\end{theorem}

\begin{proof}
This is well-known, the idea being as follows:

\bigskip

(1) This is Cayley's theorem, with the regular representation of a finite group $G$ producing an embedding $G\subset S_N$, with $N=|G|$.

\bigskip

(2) The regular representation still exists for any finite quantum group $G$, but does not produce an embedding $G\subset S_N^+$. And the point is that the quantum Cayley theorem itself fails, with the simplest counterexample being something non-trivial, constructed in \cite{bbn}. To be more precise, the quantum Cayley theorem holds at $N\leq23$, but at $N=24$ we have a counterexample, which occurs as a split abelian extension associated to the exact factorisation $S_4 = \mathbb Z_4S_3$. See \cite{bbn}.

\bigskip

(3) This is a trivial consequence of (2).
\end{proof}

Here is another negative result, coming this time from \cite{lev}:

\begin{theorem}\label{TH3.5}
The following happen:
\begin{enumerate}
\item There are uncountably many quotients of $\mathbb Z_2*\mathbb Z_3$.

\item Thus there are uncountably many quantum subgroups of $S_5^+$.

\item Thus there are quantum permutation groups not of the form $G^+(X)$.
\end{enumerate}
\end{theorem}

\begin{proof}
Once again this is well-known, the idea being as follows:

\bigskip

(1) Conjectured to be the case by Neumann \cite{ne2}, Levin \cite{lev} showed that the modular group $\operatorname{PSL}(2,\mathbb{Z})\cong \mathbb{Z}_2*\mathbb{Z}_3$ is SQ-universal. This means that every countable group can be embedded in a quotient of $\mathbb{Z}_2*\mathbb{Z}_3$, and this implies that $\mathbb Z_2*\mathbb Z_3$ has uncountably many quotients \cite{ne1}.

\bigskip

(2) This follows from (1), by taking group dual subgroups.

\bigskip

(3) This follows from (2), the set of finite graphs being countable.
\end{proof}
This non-constructive proof is unsatisfying, and  suggests looking at special classes of quantum permutation groups to which the arguments in the proof will not apply, such as:

\bigskip

\begin{enumerate}

\item The finite quantum permutation groups. By \cite{ste}, at each value of $|G|=\dim C(G)$ there are finitely many such quantum groups.

\bigskip

\item The duals of finite groups.
\end{enumerate}

\bigskip

We will investigate in what follows such questions.

\section{Intermediate liberations of finite groups}

We discuss in what follows the following notion, mainly coming from Schmidt's series of papers \cite{sc1,sc2,sc3}, and their follow-ups:

\begin{definition}\label{DEF5.1}
We call a finite group $G$ quantum-excluding when
$$G(X)=G\implies G^+(X)=G$$
holds for any finite graph $X$.
\end{definition}
It is an open question whether there is a quantum excluding group. Groups known not to be quantum-excluding include $S_N$ ($N\geq 4$), $D_4$, and $\mathbb{Z}_2\times \mathbb{Z}_2$. The question is open for cyclic groups, $S_3$, and $A_N$.

\bigskip

In relation with our Frucht type questions, discussed in section \ref{SEC3} above, assume that a finite group $G$ has a liberation $G\subset G^\times$, genuine in the sense that $C(G^\times)$ is noncommutative. Then:

\bigskip

(1) If $G$ is known to be quantum-excluding in the sense of Definition \ref{DEF5.1}, then $G^\times$ does not have the Frucht property.

\bigskip

(2) Equivalently, if $G^\times$ is known to be of the form $G^\times=G^+(X)$, then the existence of this graph $X$ shows that $G$ cannot be quantum-excluding.

\bigskip

In view of this, the construction of  liberations of finite groups $G\subset G^\times$ is an interesting question, that can be useful in the study of quantum automorphism groups, and perhaps more generally.

\bigskip

For a liberation of a compact group $G\subset G^\times \subset U_N^+$,  the standard reconstruction of the commutative  $\mathrm{C}^*$-algebra $C(G)$ from its liberation $C(G^\times)$ is via the abelianisation $C(G^\times)/I$, with $I\subset C(G^\times)$ being the commutator ideal.

\bigskip

When $G^\times\subset S_N^+$, there is an arguably more transparent approach to finding $G$, the classical version of $G^\times$. Consider a map $\Phi$ from the state space of $C(G^\times)$ to the Birkhoff polytope of $N\times N$ doubly stochastic matrices:
$$\Phi(\varphi)=\left[\varphi(u_{ij})\right]_{i,j=1}^N.$$
This map was used by Woronowicz to show that commutative Woronowicz algebras are algebras of continuous functions on compact groups $G\subset U_N$ \cite{wo1}.

\bigskip

Where $\jmath$ maps a permutation in $S_N$ to its permutation matrix in $M_N(\mathbb{C})$, each character $\omega:C(G^\times)\rightarrow \mathbb{C}$ is associated bijectively with some $\sigma\in S_N$ such that
$$\Phi(\omega)=\jmath(\sigma),$$
in which case we write $\omega=\operatorname{ev}_{\sigma}$. These characters form a group with identity $\varepsilon$ and group law given by the convolution:
$$\operatorname{ev}_{\sigma_1}\star \operatorname{ev}_{\sigma_2}=(\operatorname{ev}_{\sigma_1}\otimes \operatorname{ev}_{\sigma_2})\Delta,$$
such that:
\begin{equation}\label{EQPHI}\Phi(\operatorname{ev}_{\sigma_1}\star \operatorname{ev}_{\sigma_2})=\Phi(\operatorname{ev}_{\sigma_1})\Phi(\operatorname{ev}_{\sigma_2})=\Phi(\operatorname{ev}_{\sigma_1\sigma_2}).\end{equation}
This finite group  is the same as the classical version $G_{\text{class}}$ coming from the abelianisation $C(G^\times)/I$.
  When $G^\times$ is finite and so $C(G^\times)$ is a multi-matrix algebra,  the one-dimensional summands correspond precisely to the characters. See \cite{mcc} for more.

\bigskip

  We recall that for any finite group $G$ we have a Peter--Weyl decomposition, as follows:
$$C(\widehat{G})=\bigoplus_{r\in \operatorname{Irr}(G)}M_{\dim(r)}(\mathbb C).$$
We are interested in what follows in the one-dimensional representations, of which there is of course at least the trivial representation $G\to\{1\}$, corresponding to the counit $\varepsilon:C(\widehat{G})\rightarrow \mathbb{C}$.

\bigskip

Using these ideas,  we have the following elementary result:

\begin{proposition}
Assume that a finite group $G$ has $N$ one-dimensional representations. Then the classical version $\widehat{G}_{\text{class}}$ is an abelian group of order $N$.
\end{proposition}

\begin{proof}
The classical version $\widehat{G}_{\text{class}}$ is indeed abelian because $C(\widehat{G})^*\simeq C(G)$, and therefore the convolution of characters on $C(\widehat{G})$ is commutative.
\end{proof}

\begin{theorem}
Any finite group $G$ has a liberation.\label{PROPLIB}
\end{theorem}

\begin{proof}
For small cyclic groups, by looking at the relevant degrees of the representations,  we obtain liberations as follows:
$$\mathbb Z_1\subset\widehat{A}_5,$$
$$\mathbb Z_2\subset\widehat{S}_3,$$
$$\mathbb Z_3\subset\widehat{A}_4.$$
In general, we will make use of a perfect group, which is by definition a group that equals its commutator subgroup. Such a group has a unique representation of degree one, and therefore the classical version of its dual is trivial.

\bigskip

In practice we can take $A_5$, which is the smallest perfect group which will produce  liberations, via our procedure. If we take:
$$A_5=\langle(123),(345)\rangle$$
we have an embedding $\widehat{A}_5\subset S_6^+$.

\bigskip

We embed $G\subset S_N$, say via the Cayley theorem, and we consider a non-trivial finite perfect group $H=\langle h_1,\ldots,h_k\rangle $. Let us set:
$$M=\operatorname{ord}(h_1)+\ldots+\operatorname{ord}(h_k).$$

Consider the tensor product of $C(G)$ and $C(\widehat{H})$, which is given by:
$$C(G)\otimes C(\widehat{H})=C(G\times \widehat{H}).$$

This construction gives a  quantum permutation group $G\times\widehat{H}\subset S_{N+M}^+$, with fundamental magic representation as follows:
$$u^{G\times \widehat{H}}=
\begin{pmatrix}
u^G\otimes1_{\widehat{H}}&0 \\
0&1_G\otimes u^{\widehat{H}}
\end{pmatrix}.$$

Using $M_{N_1}(\mathbb{C})\otimes M_{N_2}(\mathbb{C})\simeq M_{N_1\times N_2}(\mathbb{C})$, note that as $C(G)$ has $|G|$ one-dimensional factors, and $C(\widehat{H})$ has a single one-dimensional factor, $C(G\times \widehat{H})$ has $|G|$ one-dimensional factors, and thus $|G|$ characters.

\bigskip

It remains to show the group of characters coincides with $G$. Characters on $C(G)$ are of the form $\operatorname{ev}_\sigma$:
$$\operatorname{ev}_\sigma:C(G)\rightarrow \mathbb{C},\qquad f\mapsto f(\sigma).$$
There is a single character on $C(\widehat{H})$, necessarily the counit, $\operatorname{ev}_e$.

\bigskip

The tensor product of $\operatorname{ev}_{\sigma}$ on $C(G)$ and $\operatorname{ev}_e$ on $C(\widehat{H})$ is a character on $C(G\times\widehat{H})$, and we have
$$\Phi(\operatorname{ev}_{\sigma}\otimes \operatorname{ev}_e)=
\begin{pmatrix}
\jmath(\sigma)&0\\
0&I_M
\end{pmatrix},$$
and via (\ref{EQPHI}) we see that the group law on the classical version of $G\times \widehat{H}$ is the same as that of $G$, as desired.
\end{proof}
The above argument is deliberately  ad-hoc for $N\leq 3$ as $\widehat{A}_5$, $\widehat{S}_3$, $\widehat{A}_4$ provide small finite quantum groups to study in the context of quantum automorphisms.

\section{Finite graphs and orbitals}\label{SEC5}
In this section we connect the theory of orbitals with an old observation of Bichon. The complement of $E$ is taken with respect to $V\times V$ minus the diagonal elements:
$$\{(v_1,v_2):v_1,v_2\in V,v_1\neq v_2\}.$$

\begin{proposition}
Suppose that $G\subset S_N^+$. If $X$ is a graph on $N$ vertices, $du=ud$ is equivalent to
\begin{align*}(i,k)\in E,\,(j,l)\in E^c &\implies u_{ij}u_{kl}=0,\text{ and}
\\ (i,k)\in E^c,\,(j,l)\in E& \implies u_{ij}u_{kl}=0.\end{align*}
\end{proposition}
\begin{proof}
  This goes back to Bichon \cite{bi2}. A recent presentation is (Prop. 3.5, \cite{ch1}).
\end{proof}
Classically this can be understood as follows: suppose that $(j,l)\in E$ and $(i,k)\in E^c$:

\medskip

\begin{center}
\begin{tikzcd}
\substack{l\\ \bullet} \arrow[d, no head] & \substack{k\\ \bullet} \\
\substack{\bullet\\ j}                    & \substack{\bullet\\ i}
\end{tikzcd}
\end{center}

\medskip

 If a  graph automorphism $\sigma:V\rightarrow V$ maps $l\rightarrow k$,  it cannot also map $j\rightarrow i$ as this would violate
 $$(j,l)\in E\implies (\sigma(j),\sigma(l))=(i,k)\in E.$$

\bigskip

We will make great use of the contrapositive:
\begin{proposition}\label{Bichobs}
Suppose that $G\subset S_N^+$.  If $X$ is a graph on $N$ vertices, $du=ud$ is equivalent to
$$u_{ij}u_{kl}\neq 0\quad \implies \quad (i,k),(j,l)\in E\text{ or }(i,k),(j,l)\in E^c.$$
\end{proposition}

\bigskip

 The following is contained in \cite{lmr} and is related to the notion of a coherent configuration. The diagonal relation is on $\{1,\dots,N\}$:

\begin{theorem} Orbitals of $G\subset S_N^+$ are either subsets of the diagonal relation (corresponding to orbits), or disjoint of the diagonal relation. For each orbital
$$o=\{(i,k)\colon (i,k)\in o\},$$
its inverse
$$o^{-1}=\{(k,i)\colon (i,k)\in o\}$$
is also an orbital (and it can happen that $o=o^{-1}$).
\end{theorem}
By labelling the vertices $V=\{1,2,\dots,N\}$, if $G\curvearrowright X$, the orbitals of $G\subset S_N^+$ are subsets of $V\times V$.
\begin{proposition}\label{orbitalgraph}
Suppose that $G\curvearrowright X$. We have, where $o\subset V\times V$ is the orbital of $(i,k)$,
$$(i,k)\in E\implies (o\cup o^{-1})\subset E.$$
Similarly
$$(i,k)\in E^c\implies (o\cup o^{-1})\subset E^c.$$
\end{proposition}
\begin{proof}
This follows from the definition of an orbital and Proposition \ref{Bichobs}: suppose that $(i,k)\in E$ and $(j,l)\in o$. Then
$$u_{ij}u_{kl}\neq 0\implies (j,l)\in E\implies o\subset E.$$
With finite graphs assumed symmetric, $o\subset E\implies o^{-1}\subset E$. The argument is the same if we suppose $(i,k)\in E^c$ and $(j,l)\in o$.
\end{proof}

\bigskip

As the edge relation is irreflexive,  the diagonal orbitals are subsets of $E^c$. The next result provides us with a complete list of the finite graphs on $N$ vertices that $G\subset S_N^+$ acts on:

\begin{theorem}\label{orbitalgraph2}
Let $O$ be the set of orbitals of $G\subset S_N^+$ disjoint of the diagonal relation. For each  $A\subset O$ define a graph $X_A=(\{1,\dots,N\},E_A)$ where
$$E_A:=\bigcup_{o\in A}(o\cup o^{-1}).$$
Then $G\curvearrowright X_A$. Furthermore, if $G\curvearrowright X$ for $X$ a graph on $N$ vertices, then $X=X_A$ for some $A\subset O$.
\end{theorem}
\begin{proof}
Suppose  that $u_{ij}u_{kl}\neq 0$, so that $(i,k)\sim (j,l)$ are in the same orbital, say $o'$. By construction of $X_A$, if $o'\in A$, both   $(i,k),\,(j,l)\in E_A$; and if $o'\in A^c$, then both $(i,k),\,(j,l)\in E_A^c$. Thus by Proposition \ref{Bichobs},  $G\curvearrowright X_A$.

\bigskip

That the set of graphs $\{X_A:A\subset O\}$ exhausts the graphs on $N$ vertices that $G\subset S_N^+$ acts on is Proposition \ref{orbitalgraph}.
\end{proof}

\begin{theorem}\label{orbitalgraph3}
Suppose $G\curvearrowright X$ with $O$ the set of orbitals disjoint of the diagonal relation. Suppose $\sigma_1,\dots,\sigma_k\in S_V$ leave the orbitals invariant: for all $\sigma_p\in S_V$ and $o\in O$,
$$(i,j)\in o\implies (\sigma_p(i),\sigma_p(j))\in o.$$
Then each $\sigma_p\in G(X)$, and so where $H=\langle\sigma_1,\dots,\sigma_k\rangle$,
$$H\subset G(X).$$
Finally
$$H \not\subset G_{\text{class}}\implies G\subsetneq G^+(X).$$
\end{theorem}
\begin{proof}
We have:
$$G= G^+(X)\implies G_{\text{class}}=G(X),$$
and so if $H \not\subset G_{\text{class}}$, we cannot have $G=G^+(X)$.  The rest follows immediately from Theorem \ref{orbitalgraph2}
\end{proof}

\bigskip

As an illustration, consider the quantum hyperoctahedral group $H_2^+\subset S_4^+$ freely generated by the magic unitary (the conventional presentation, from \cite{bbc}, is got from swapping labels two and three):
$$u=\begin{pmatrix}
              p_{11} & q_{11}  & p_{12} & q_{12} \\
              q_{11} & p_{11} & q_{12} & p_{12} \\
              p_{21} & q_{21} & p_{22} & q_{22} \\
              q_{21} & p_{21} & q_{22} & p_{22}
            \end{pmatrix}$$
Suppose $H_2^+\subset S_4^+$ acts on a finite graph $X$. By finding all $u_{ij}u_{kl}\neq 0$ we see that the orbitals disjoint of the diagonal relation are:
\begin{align*}
  o_1 & =\{(1,3),(1,4),(2,3),(2,4),(3,1),(3,2),(4,1),(4,2)\}, \\
  o_2 & =\{(1,2),(2,1),(3,4),(4,3)\}:
\end{align*}
\adjustbox{scale=1,center}{%
$$\begin{tikzcd}
\substack{1\\\bullet} \arrow[rd,no head] & \substack{2\\ \bullet} \arrow[ld,no head] & {} \arrow[d, no head] & \substack{1\\ \bullet} \arrow[r,no head] & \substack{2\\ \bullet} \\
\substack{\bullet\\ 4} \arrow[u,no head]  & \substack{\bullet\\ 3} \arrow[u,no head]  & {}                    & \substack{\bullet\\ 4} \arrow[r,no head] & \substack{\bullet\\ 3}
\end{tikzcd}$$
}

\bigskip

Where $s=(14)(23)$ is the `switch', and $b_i\in \{0,1\}$, each
$$\sigma_{(b_1,b_2,b_3)}=(34)^{b_3}(12)^{b_2}s^{b_1},$$
leaves the orbitals invariant.  It is the case that $\langle\sigma_{(b_1,b_2,b_3)}\rangle\cong D_4$, and it follows that for all graphs $X$ on four vertices such that $H_2^+\curvearrowright X$,
$$D_4\subset G(X).$$
Theorem \ref{orbitalgraph2} gives all these graphs:

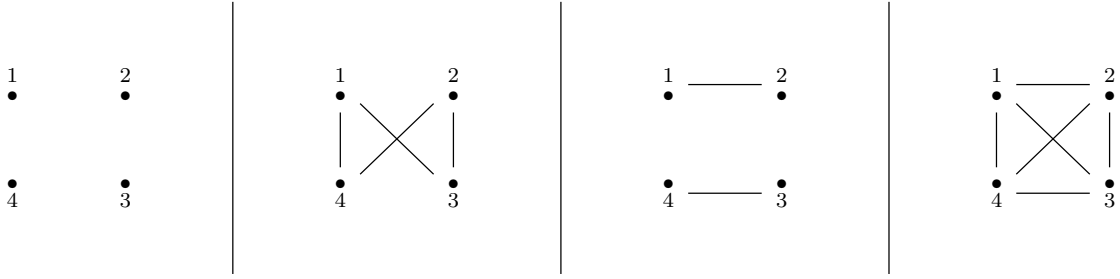
\adjustbox{scale=1,center}{%
$$\begin{tikzcd}
  &   & {} \arrow[ddd, no head] &              &              & {}                      &             &   & {}                      &                                                             &                      \\
\substack{1\\ \bullet} & \substack{2\\ \bullet} &                         & \substack{1\\ \bullet} \arrow[rd,no head] & \substack{2\\ \bullet} \arrow[ld,no head] &                         & \substack{1\\ \bullet} \arrow[r,no head] & \substack{2\\ \bullet} &                         & \substack{1\\ \bullet} \arrow[r, no head] \arrow[d, no head] \arrow[rd, no head] & \substack{2\\ \bullet} \arrow[d, no head] \\
\substack{\bullet\\ 4} & \substack{\bullet\\ 3} &                         & \substack{\bullet \\ 4} \arrow[u,no head]  & \substack{\bullet\\ 3} \arrow[u,no head]  &                         & \substack{\bullet \\ 4} \arrow[r,no head] & \substack{\bullet \\ 3} &                         & \substack{\bullet \\ 4} \arrow[r, no head] \arrow[ru, no head]                    & \substack{\bullet\\ 3}                    \\
  &   & {}                      &              &              & {} \arrow[uuu, no head] &             &   & {} \arrow[uuu, no head] &                                                             &
\end{tikzcd}$$
}
Indeed we know from \cite{bi2} that  $H_{2,\text{class}}^+=D_4$, $G^+(\Box)=H_2^+$, and so $H_2^+$ has the Frucht property.

\bigskip

\begin{proposition}
Suppose that $G\curvearrowright X$ is given by a block magic unitary
$$u=\begin{pmatrix}
      u^1 &  & \\
       & \ddots &  \\
       &  & u^k
    \end{pmatrix}.$$
    with associated blocks $V_1,\dots,V_k$ forming a partition of $V$. Each orbital is a subset of some $V_p\times V_q$.
\end{proposition}
\begin{proof}
This follows from the fact that the orbits of $G$ are  blocks, $V_p$.
\end{proof}

We will use the following argument to show that a quantum permutation group $G$ does not have the Frucht property:

\bigskip

(1) We assume that an arbitrary embedding $G\subset S_N^+$ acts on a finite graph $X$.

\bigskip

(2) With $G\curvearrowright X$, the permutation $P$ from Proposition \ref{PROPBLO} can be taken to be a relabelling of $V$ so that $u$ is block diagonal with transitive magic representations:
$$u=\begin{pmatrix}
      u^1 &  &  \\
       & \ddots &  \\
       &  & u^k
    \end{pmatrix}.$$
    Furthermore, this relabelling of $V$ can be chosen so that, up to multiplicity, at most one representative from each equivalence class of transitive magic representations appears. We can assume furthermore that transitive magic representations of the same type occurring with multiplicity appear consecutively along the diagonal, i.e.
    $$\operatorname{diag}(u^1,u^1,u^2,u^3,u^{3})\text{ rather than e.g. }\operatorname{diag}(u^3,u^2,u^1,u^1,u^{3})$$

    \bigskip

The $m$-th appearance of $u^p$ along the diagonal is associated with a block of vertices $V_p^{(m)}\subset \{1,2,\dots,N\}=V$, and we have a partition
$$V=\{1,2,\dots,N\}=\bigsqcup_{p,m}V_p^{(m)}.$$

\bigskip

To illustrate the notation, suppose that $G\subset S_7^+$ acts on a graph. Suppose  $u^1\in M_{2}(C(G))$, $u^2\in M_{3}(C(G))$, and $u=\operatorname{diag}(u^1,u^1,u^2)$. This partitions $$V=\{1,2\}\sqcup\{3,4\}\sqcup\{5,6,7\},$$ and we denote $V_1^{(1)}=\{1,2\}$, $V_1^{(2)}=\{3,4\}$, and $V_2^{(1)}=\{5,6,7\}$.

\bigskip

(3)  We find, for each $(p,q)\in \{1,\dots,k\}^2$, the $V_p\times V_q$ orbitals of $G\subset S_N^+$.

\bigskip

(4) We  find  a  large (classical) group of  $H$ of permutations of $V$ that leaves both the blocks, and moreover, the orbitals invariant. Using Theorem \ref{orbitalgraph3}, this is an implication:

$$G\curvearrowright X\implies H \curvearrowright X.$$

\bigskip

(5) If $H \not\subset G_{\text{class}}$, by Theorem \ref{orbitalgraph3},
$$G^+(X)\neq G,$$ and as $X$ was chosen arbitrarily from all graphs such that $G\curvearrowright X$, we conclude that $G$ does not have the Frucht property.

\bigskip

With the notation explained, two helpful propositions:
\begin{proposition}\label{ProDia}
Where $G\subset S_N^+$ with blocks of type $V_p$ of size $N_p$, and $O$  the set of orbitals disjoint of the diagonal relation,  the orbitals of  $V_{p}^{(m)}\times V_{p}^{(m+n)}$ are of the form:
  \begin{align*}
    o' &:=\{(i,j+nN_p)\colon  o\in O,\,(i,j)\in o\}  \\
    o_d &=\{(i,i+nN_p)\colon i\in V_{p}^{(m)}\}
  \end{align*}
\end{proposition}
\begin{proof}
Implicit here is that the labelling of $V$ has been chosen with all the $V_p$ blocks together:
$$u=\operatorname{diag}(\dots,u^p,\dots,u^p,\dots).$$
This means that
$$u_{ij}u_{kl}\neq 0\implies u_{ij}u_{k+nN_p,l+nN_p}\neq 0.$$
Finally the diagonal orbital in $V_p^{(m)}\times V_p^{(m)}$, ordinarily not in play with finite graphs assumed without loops, comes into play via:
$$u_{ij}u_{i+nN_p,j+nN_p}\neq 0.$$
\end{proof}
\begin{proposition}\label{PROPCheck1}
Suppose $G\curvearrowright X$. If $\sigma\in S_V$ leaves all but one orbital (disjoint of the diagonal relation) invariant, it is  invariant for all of them.
\end{proposition}
\begin{proof}
If $V_p\neq V_q$ there is no diagonal relation. Orbitals form a partition of $V_p\times V_q$ and permutations $\sigma\in S_{V}$ induce a permutation $\sigma'$ on $V_p\times V_q$:
$$(i,j)\mapsto (\sigma(i),\sigma(j)).$$
If there are $n$ orbitals $o_1,\dots,o_n$, and $o_1,\dots,o_{n-1}$ are invariant under $\sigma'$, then it follows as $\sigma'$ is a permutation that $\sigma'(o_n)=o_n$.

\bigskip

If $V_p=V_q$, there is a diagonal orbital, $o_d$, but the induced permutation $\sigma'$ leaves the diagonal relation invariant.
\end{proof}
\section{Finite duals}\label{SECDuals}
If a group dual $\widehat{\Gamma}\subset S_N^+$ acts on a finite graph $X$, with appropriate vertex labelling, Bichon's representation results from \cite{bi3} implies that  $u$ is a block diagonal matrix of Fourier-type transitive magic representations $u^g$.  Where each $g_p\in\Gamma$ occurs in $u$ with multiplicity $m_p\geq 0$, and $N_p:=\operatorname{ord}(g_p)$, the group dual action $\widehat{\Gamma}\curvearrowright X$ partitions
 $$V:=\bigsqcup_{g_p\in\Gamma}\left(\bigsqcup_{m=0}^{m_{p}}V_{p}^{(m)}\right)=\bigsqcup_{g_p\in \Gamma}\left(\bigsqcup_{m=0}^{m_p}\{i^{(m)}_{p},i^{(m)}_{p}+1,\dots,i^{(m)}_p+N_p-1\}\right).$$
As explained above, $V_p^{(m)}=\{i^{(m)}_{p},i^{(m)}_{p}+1,\dots,i^{(m)}_p+N_p-1\}$ corresponds to the vertices given by the $m$-th copy of $u^{g_p}$ in the block diagonal magic unitary $u$.

\bigskip

\begin{definition}
If $V_p,V_q\subset V$, and $V_p\times V_q\subset E$ or $V_p\times V_q\subset E^c$, say that the edge relation between $V_p$ and $V_q$  is total. If $V_p=V_q$, say that the edge relation is total if for all $i\neq j$, $(i,j)\in E$ or $(i,j)\in E^c$. 

\bigskip

 Furthermore say that a finite  group dual $\widehat{\Gamma}$ is total if for all $g_p,g_q\in \Gamma$, $1\leq i,j\leq N_p$, and $1\leq k,l\leq N_q$
$$u^{g_p}_{ij}u^{g_q}_{kl}=0\implies u^{g_p},u^{g_q} \text{ are equivalent transitive magic representations}.$$
\end{definition}

\begin{proposition}\label{totalprop}
If total $\widehat{\Gamma} \curvearrowright X$, then the edge relation between inequivalent blocks is total.
\end{proposition}
\begin{proof}
Let $g_p$ of order $N_p$ and $g_q$ of order $N_q$ give inequivalent transitive magic representations.  Let $i,j\in V_{p}$ and $k,l\in V_{q}$:
$$u_{ij}^{g_p}u_{kl}^{g_q}\neq 0\implies (i,k)\sim (j,l)\implies V_{p}\times V_{q}\text{ is an orbital}.$$ This implies via Theorem \ref{orbitalgraph2} that the relations between inequivalent blocks are total.
\end{proof}
This theorem breaks the convention that $k$ denotes the number of blocks in $G\subset S_N^+$. By Proposition \ref{PROPBLO}, we can assume that at most one representative from each equivalence class of transitive magic representations of $C(\widehat{\Gamma})$ appears in $\widehat{\Gamma}\subset S_N^+$, and $k$ here represents the number of equivalence classes appearing:
\begin{theorem}\label{total}
Suppose that total $\widehat{\Gamma}\curvearrowright X$. Let $\Gamma_u=\{g_1,\dots,g_k\}$  be the set of  elements of $\Gamma$ appearing as Fourier-type transitive magic representations in the fundamental representation.  Then, where $g_1,\dots,g_k$ are of order $N_1,\dots,N_k$,
$$\mathbb{Z}_{N_1}\times\cdots \mathbb{Z}_{N_k}\curvearrowright X.$$
If $N_p=3$ and $m_{p}=1$, the copy of $\mathbb{Z}_{N_p}$ can be replaced by $S_3$.
\end{theorem}
\begin{proof}
Let $g_p\in\Gamma_u$. Due to the circulant nature of $u^{g_p}$, with the (lower) index $N_p$ identified with $0$ in the (lower) indices of $u^{g_p}$, for any $i,j\in \{1,\dots,N_p\}$ and $l\in\{0,\dots,N_p-1\}$:
$$u^{g_p}_{ij}=u^{g_p}_{i+l,j+l}\implies u^{g_p}_{ij}u^{g_p}_{i+l,j+l}\neq 0\implies (i,i+l)\sim (j,j+l).$$
This gives a $V_{p}^{(a)}\times V_{p}^{(a)}$ orbital of the form
$$o_l=\{(i,j)\,\colon\, j-i\mod N_p=l\},$$
with $o_0$ giving the diagonal relation.

\bigskip

Let us consider now $V_{p}^{(a)}\times V_{p}^{(b)}$ orbitals.  Proposition \ref{ProDia} gives, for $l\in\{0,\dots,N_p-1\}$, $V_{p}^{(a)}\times V_{p}^{(b)}$ orbitals of the form:
\begin{equation}\label{edgerel1}
o_l^{(a,b)}=\{(i^{(a)}+j_1,i^{(b)}+j_2)\,\colon\, j_2-j_1\mod N_p=l\}.
\end{equation}

\bigskip

By Proposition \ref{totalprop}, for blocks of $V_{p}$ and $V_{q}$ of different type, the unique $V_{p}\times V_{q}$ orbital is $V_{p}\times V_{q}$ itself.

\bigskip

For each $g_p\in \Gamma_u$, fix $l_p\in \{0,\dots,N_p-1\}$. Consider the permutation $\sigma_{(l_1,l_2,\dots,l_{k})}$ that maps
$$i^{(m)}_{p}+j\rightarrow i^{(m)}_{p}+(j+l_p\mod N_p).$$
Note that
\begin{equation}\label{restrict}\sigma_{(l_1,l_2,\dots,l_{k})}(V_{p}^{(m)})=V_{p}^{(m)}.\end{equation}

The edge relation between a $V_{p}^{(m)}$ and an inequivalent block is total, so any permutation of $V_{p}$ will leave the relations there invariant. It remains to show that the permutation above leaves invariant the orbitals between a pair of possibly equal equivalent blocks $V^{(a)}_{g_p},\,V^{(b)}_{g_p}$.

\bigskip

Let
$$(i^{(a)}_{p}+j_1,i^{(b)}_{p}+j_2)\in o^{(a,b)}_l\implies j_2-j_1\mod N_p=l.$$
Then $(\sigma_{(l_1,l_2,\dots,l_{k})}(i^{(a)}_{p}+j_1),\sigma_{(l_1,l_2,\dots,l_{k})}(i^{(b)}_{p}+j_2))$ is equal to
$$(i^{(a)}_{p}+(j_1+l_p\mod N_p),i^{(b)}_{p}+(j_2+l_p\mod N_p)).$$
If $j_2-j_1\mod N_p=l$ then
$$((j_2+l_p\mod N_p)-(j_1+l_p\mod N_p))\mod N_p=l$$
$$\implies (\sigma_{(l_1,l_2,\dots,l_{k})}(i^{(a)}_{p}+j_1),\sigma_{(l_1,l_2,\dots,l_{k})}(i^{(b)}_{p}+j_2))\in o_l^{(a,b)}$$
and so by Theorem \ref{orbitalgraph3}  the abelian group
$$\langle\sigma_{(l_1,l_2,\dots,l_k)} \colon l_p\in \{0,\dots, N_p-1\}\rangle\cong \mathbb{Z}_{N_1}\times\cdots \mathbb{Z}_{N_k}$$
 is a subgroup of $G(X)$.

\bigskip

If $N_p=3$ and $m_{p}=1$, then the induced subgraph of $V_{p}=\{i^{(1)}_{p},i_{p}^{(1)}+1,i_{p}^{(1)}+2\}$ is complete or edgeless, and it follows that the copy of $\mathbb{Z}_{N_p}$ can be taken to be $S_3$.
\end{proof}
As an illustration of Theorem \ref{total}, consider the total finite dual $\widehat{S_3}$ with $\widehat{S_3}\subset S_5^+$ given by $u=\operatorname{diag}(u^{(12)},u^{(123)})$. There are five orbitals disjoint of the diagonal relation:
\begin{align*}
  o_1 & =\{(1,2),(2,1)\} \\
  o_2 & =\{1,2\}\times\{3,4,5\}
  \\ o_2^{-1}&=\{3,4,5\}\times\{1,2\}
  \\o_3&=\{(3,4),(4,5),(5,3)\}
  \\ o_3^{-1}&=\{(3,5),(5,4),(4,3)\}
\end{align*} Theorem \ref{orbitalgraph2} thus gives eight graphs on five vertices  that $\widehat{S_3}$ acts on, and each admits a $\mathbb{Z}_2\times S_3$ action:

\bigskip

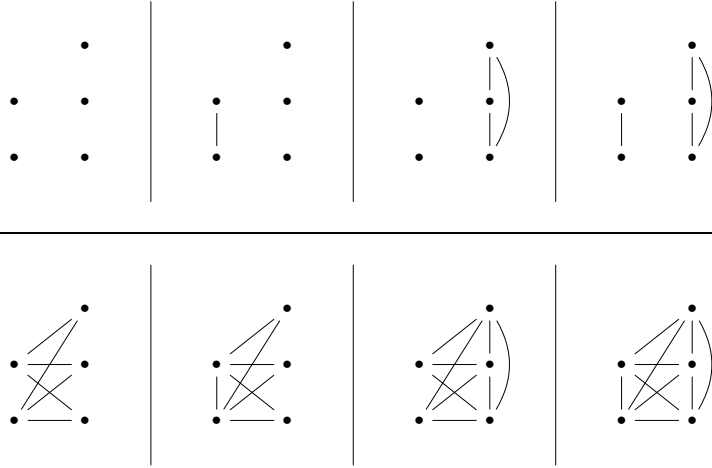
\adjustbox{scale=0.6,center}{%
\begin{tikzcd}
        &         & {} \arrow[dddd,no head] &                   &         & {} \arrow[dddd,no head] &         &                                         & {} \arrow[dddd,no head] &                   &                                         \\
        & \bullet &                 &                   & \bullet &                 &         & \bullet \arrow[d,no head] \arrow[dd, bend left,no head] &                 &                   & \bullet \arrow[d,no head] \arrow[dd, bend left,no head] \\
\bullet & \bullet &                 & \bullet \arrow[d,no head] & \bullet &                 & \bullet & \bullet \arrow[d,no head]                       &                 & \bullet \arrow[d,no head] & \bullet \arrow[d,no head]                       \\
\bullet & \bullet &                 & \bullet           & \bullet &                 & \bullet & \bullet                                 &                 & \bullet           & \bullet                                 \\
        &         & {}              &                   &         & {}              &         &                                         & {}              &                   &
\end{tikzcd}
}

\medskip

\rule{\textwidth}{0.4pt}

\medskip

\adjustbox{scale=0.6,center}{%
\begin{tikzcd}
                                         &         & {} \arrow[dddd, no head] &                                                            &         & {} \arrow[dddd, no head] &                                          &                                                           & {} \arrow[dddd, no head] &                                                            &                                                           \\
                                         & \bullet &                          &                                                            & \bullet &                          &                                          & \bullet \arrow[d, no head] \arrow[dd, no head, bend left] &                          &                                                            & \bullet \arrow[d, no head] \arrow[dd, no head, bend left] \\
\bullet \arrow[ru, no head] \arrow[r, no head] \arrow[rd, no head]  & \bullet &                          & \bullet \arrow[d, no head] \arrow[ru, no head] \arrow[r, no head] \arrow[rd, no head] & \bullet &                          & \bullet \arrow[ru, no head] \arrow[r, no head] \arrow[rd, no head]  & \bullet \arrow[d, no head]                                &                          & \bullet \arrow[d, no head] \arrow[ru, no head] \arrow[r, no head] \arrow[rd, no head] & \bullet \arrow[d, no head]                                \\
\bullet \arrow[r, no head] \arrow[ru, no head] \arrow[ruu, no head] & \bullet &                          & \bullet \arrow[r, no head] \arrow[ru, no head] \arrow[ruu, no head]                   & \bullet &                          & \bullet \arrow[ruu, no head] \arrow[ru, no head] \arrow[r, no head] & \bullet                                                   &                          & \bullet \arrow[r, no head] \arrow[ru, no head] \arrow[ruu, no head]                   & \bullet                                                   \\
                                         &         & {}                       &                                                            &         & {}                       &                                          &                                                           & {}                       &                                                            &
\end{tikzcd}
}

\bigskip

\begin{proposition}
Let $p$ be a prime. If $\Gamma$ is non-abelian, has $p$ one dimensional representations, and $\widehat{\Gamma}$ is total, then $\widehat{\Gamma}$ does not have the Frucht property.
\end{proposition}
\begin{proof}
If $\Gamma$ is non-abelian it is not cyclic, and so any generating set has more than two elements.  If $\widehat{\Gamma}$ is total, then by Theorem \ref{total}, a non-trivial abelian group with a composite number of elements is a subgroup of $G(X)$. However if $G^+(X)=\widehat{\Gamma}$, then $G(X)=\widehat{\Gamma}_{\text{class}}\cong \mathbb{Z}_p$ which has no non-trivial subgroups.
\end{proof}
\begin{theorem}
The finite duals $\widehat{S_3}$, $\widehat{A_4}$, and $\widehat{A_5}$ do not have the Frucht property.
\end{theorem}
\begin{proof}
We observed in Proposition \ref{PROPLIB} that $S_3$, $A_4$, $A_5$ have, respectively, two, three, and one, one-dimensional representations.

\bigskip

We claim that each of these finite duals is total. A sufficient condition for totality of $\widehat{\Gamma}$ is that for inequivalent transitive magic representations $u^{g_p},\,u^{g_q}$   the $N_p\cdot N_q$ matrix with entries in $C(\widehat{\Gamma})$:
\begin{equation}\label{g1g2matrix}
K(g_p,g_q):=\left[g_p^mg_q^n\right]_{m=1,\dots,N_p,\,n=1,\dots,N_q},
\end{equation}
contains a unique entry, say $g^*$. This implies that
$$u_{ij}^{g_p}u_{kl}^{g_q},$$
a formal sum in $\Gamma$, contains a single non-zero multiple of $g^*$, and so cannot be zero. Note that by writing down the Fourier-type transitive magic representation $u^{g_p}$, we can see that if $g_p,\,g_q\in \Gamma$ are of both of order three or both of order five, and  $\langle g_p\rangle=\langle g_q\rangle$, then the Fourier-type transitive magic representations $u^{g_p}$ and $u^{g_q}$ are equivalent.

\bigskip

In $\widehat{S_3}$ the only pair of equivalent transitive magic representations are those given by $(123)$ and $(132)$. It follows quickly that $\widehat{S_3}$ is total.

\bigskip

In $A_4$ there are eight three-cycles, three disjoint two-cycles, and the identity, giving eight inequivalent representations. If $\tau_1,\, \tau_2$ are inequivalent disjoint two-cycles, then both are unique entries in $K(\tau_1,\tau_2)$. Given inequivalent three-cycles $\sigma_1,\,\sigma_2$, as $\sigma_1^{-1}=\sigma_1^2$ and $(\sigma_1^2)^{-1}=\sigma_1$, the identity is a unique entry in $K(\sigma_1,\sigma_2)$ with a similar situation for $K(\tau_1,\sigma_1)$.

\bigskip

In $A_5$ there are 24 five-cycles (six equivalence classes $f_1,\dots f_6$), 20 three cycles (10 equivalence classes $\sigma_1,\dots,\sigma_{10}$), 15 disjoint two-cycles $\tau_1,\dots,\tau_{15}$, and the identity. Assuming inequivalent representations, similarly to $A_4$, the identity is a unique entry in each of
$$K(f_a,f_b),\,K(\sigma_c,\sigma_d),\,K(\tau_k,\tau_l).$$
For the mixed cases note that the inverse of any $k$-cycle is a $k$-cycle, and the inverse of the disjoint two-cycle is itself. It follows that for any $K$ with inputs of inequivalent type the identity is a unique entry.
\end{proof}

The smallest examples of  non-total non-abelian group duals are the duals of the dihedral group, where for example, $u^{r}_{11}u^{r^2}_{21}=0$, and the quaternion group, where, for example, $u_{11}^ju_{21}^k=0$.  Both these finite duals have classical version $\mathbb{Z}_2\times\mathbb{Z}_2$.

\bigskip

Perfect groups $\Gamma$ provide intermediate liberations $\mathbb{Z}_1\subset \widehat{\Gamma}$ that are potential counterexamples to $\mathbb{Z}_1$ is quantum-excluding. The previous results show that if perfect and total $\widehat{\Gamma}\curvearrowright X$, then $G(X)\neq \mathbb{Z}_1$. The binary icosahedral group $2I$ is an example of a perfect group which is not total. It is an order 120 subgroup of the unit quaternions, and has order six and order 10 generators:
$$s=\frac12(1+i+j+k)\text{ and }t=\frac12(\varphi1+\varphi^{-1}i+j).$$
Here $\varphi=\frac{1+\sqrt{5}}{2}$ is the golden ratio. Direct calculation shows that, for example:
$$u^{s}_{11}u^t_{51}=0.$$
Perhaps speculatively, such a dual is deserving of further study. With it being non-total, Theorem \ref{total} is not a barrier to $\widehat{2I}$ having the Frucht property. If $\widehat{2I}$ does have the Frucht property, a graph $X$ such that $G^+(X)=\widehat{2I}$ is a graph with quantum symmetries but with trivial (classical) automorphism group. Furthermore, such a graph will have quantum symmetries with  a finite quantum automorphism group.

\bigskip

Even if $\widehat{2I}$ does not have the Frucht property, if $\widehat{2I}\curvearrowright X$ such that $G(X)=\mathbb{Z}_1$, then $X$ is a graph with quantum symmetries but with trivial (classical) automorphism group.

\section{The Kac--Paljutkin quantum group}\label{SECKP}

\subsection{Magic Representation Theory}
The Kac--Paljutkin quantum group of order eight \cite{kpa}, $G_0$, has algebra of functions structure:
\begin{equation}C(G_0)=\mathbb{C}\oplus\mathbb{C}\oplus\mathbb{C}\oplus\mathbb{C}\oplus M_2(\mathbb{C}).\label{KP1}\end{equation}
The choice to denote this finite quantum group by $G_0$  harks back to the original notation of \cite{kpa}, and the zero reminds this is a `ground zero' for compact quantum groups, the smallest one neither commutative nor cocommutative. Let us denote by $f_1,f_2,f_3,f_4$ the generators of the four copies of $\mathbb{C}$ and by $E_{ij}$, $1\leq i,j\leq 2$ the standard generators of $M_2(\mathbb{C})$.

\bigskip

The counit is given by $\varepsilon(f_1)=1$ (and zero elsewhere). The antipode is given by $S(f_i)=f_i$, and, where $E_{kl}$ are matrix units in the two dimensional summand, $S(E_{kl})=E_{lk}$. The comultiplication is well presented in \cite{frg}. The Haar measure is given by:
$$\int_{G_0}f_i=\frac{1}{8}\text{, and }\int_{G_0}E_{kl}=\frac{\delta_{kl}}{4}.$$
Where $I_2$ is the identity in the $M_2(\mathbb{C})$ summand, and the projection
\begin{equation}\label{EQpKP}p:=\left(0,0,0,0,\left(\begin{array}{cc}
                    \frac12 & \frac12 e^{-i\pi/4} \\
                    \frac12 e^{+i\pi/4} & \frac12
                  \end{array}\right)\right),\end{equation}
a concrete exhibition of $G_0\subset S_4^+$ (Th. 1.1 (8), \cite{bb3}) comes by the fundamental magic representation:
\begin{align}
u^{G_0}:=\left[\begin{array}{cccc}
                            f_1+f_2 & f_3+f_4 & p & I_2-p \\
                            f_3+f_4 & f_1+f_2 & I_2-p & p \\
                            p^T & I_2-p^T & f_1+f_3 & f_2+f_4 \\
                            I_2-p^T & p^T & f_2+f_4 & f_1+f_3
                          \end{array}\right]. \label{KPMU}
\end{align}
Another transitive magic representation $w\in M_4(C(G_0))$, whose entries do not generate $C(G_0)$, is given by:
\begin{equation}\label{EQW}w=\begin{pmatrix}
      f_1+f_4 & f_2+f_3 & E_{11} & E_{22} \\
      f_2+f_3 & f_1+f_4 & E_{22} & E_{11} \\
      E_{11} & E_{22} & f_1+f_4 & f_2+f_3 \\
      E_{22} & E_{11} & f_2+f_3 & f_1+f_4
    \end{pmatrix}.\end{equation}
There are three transitive magic representations $x,y,z\in M_2(G_0)$ given by
\begin{align}x&=\begin{pmatrix}
        f_1+f_2+f_3+f_4 & I_2 \\
        I_2 & f_1+f_2+f_3+f_4
      \end{pmatrix},
\\y&=\begin{pmatrix}
      f_1+f_4+E_{11} & f_2+f_3+E_{22} \\
      f_{2}+f_3+E_{22} & f_1+f_4+E_{11}
    \end{pmatrix},
\\ z&=\begin{pmatrix}
      f_1+f_4+E_{22} & f_2+f_3+E_{11} \\
      f_{2}+f_3+E_{11} & f_1+f_4+E_{22}
    \end{pmatrix}.\label{KPMU2}\end{align}

\begin{proposition}\label{ProKPB}
The full set of inequivalent transitive magic representations of $C(G_0)$ is
$$M=\{u^{G_0},w,x,y,z,1_{G_0}\}.$$
 If $G_0\subset S_N^+$,  then $u$ can be taken to be block diagonal, with blocks from $M$, with $u^{G_0}$ necessarily appearing.
\end{proposition}

\begin{proof}
We have to show in the first instance that there are, up to equivalence, only six transitive magic representations.

\bigskip

(1) The only magic unitary in $M_1(C(G_0))$ is $1_{\mathfrak{G_0}}$.

\bigskip

(2)  Suppose $v$ is a transitive magic representation in $M_2(C(G_0))$. Suppose the $M_2(\mathbb{C})$ factor of $v_{11}$ is zero.  As we must have by Proposition \ref{THORE} $\int_{G_0} v_{11}=1/2$, it follows that $v_{11}=x_{11}$ and so $v=x$.

\bigskip

We cannot have $v_{11}=I_2$, as we must have $\varepsilon(v_{11})=1$. Therefore assume we split $I_2=q+(I_2-q)$ into  rank one projections:
$$q=\begin{pmatrix}
      q_{11} & q_{12} \\
      q_{21} & 1-q_{11}
    \end{pmatrix}\implies \int_{G_0}q=\frac14.$$
Furthermore to satisfy $\int_{G_0} v_{11}=1/2$, we must have, where $i,j,k$ is a permutation of $\{2,3,4\}$:
$$v=\begin{pmatrix}
      f_1+f_i+q & f_j+f_k+I_2-q \\
      f_j+f_k+I_2-q & f_1+f_i+q
    \end{pmatrix}.$$

    \bigskip

    Suppose that $i\neq 4$. When we examine the comultiplication (see \cite{frg}), we note that the coefficient of $E_{11}\otimes E_{11}$ in $\Delta(f_1)$ and $\Delta(f_4)$ is $1/2$, and they are the only basis elements to have such a term in their comultiplication. To have:
$$\Delta(v_{11})=v_{11}\otimes v_{11}+v_{12}\otimes v_{21},$$
we require therefore that the coefficient of $E_{11}\otimes E_{11}$ in
$$q\otimes q+(I_2-q)\otimes(I_2-q)$$
is also $1/2$. That is we require:
$$q_{11}^2+(1-q_{11})^2=\frac12 \implies q_{11}=\frac12.$$
This implies further that
$$q=\begin{pmatrix}
      \frac12 & \frac12 e^{i\theta} \\
      \frac12 e^{-i\theta} & \frac12
    \end{pmatrix}.$$

Now we examine the coefficient of $E_{21}\otimes E_{12}$ in $\Delta(v)=v_{11}\otimes v_{11}+v_{12}\otimes v_{21}$. This is:
\begin{equation}\label{ETC}\frac12 e^{-i\theta}\frac12 e^{i\theta}+\left(-\frac12 e^{-i\theta}\right)\left(-\frac12 e^{i\theta}\right)=+\frac12.\end{equation}
However the coefficients of $E_{21}\otimes E_{12}$ in $\Delta(f_2)$ and $\Delta(f_3)$ are  $\pm i/2$. We conclude we cannot have $i\neq 4$.

\bigskip

Now suppose that $i=4$. The coefficient of $E_{11}\otimes E_{11}$ in $\Delta(f_1+f_4+q)$ is one, implying that $q=E_{11}$ or $E_{22}$. With $q=E_{11}$ we get $y$, and with $q=E_{22}$ we get $z$.

\bigskip

(3) We cannot have a transitive magic representation $u\in M_3(C(G_0))$ as we cannot achieve $\int_G u_{ij}=1/3$. Similar considerations rule out transitive magic representations of size greater than four.

\bigskip

(4) It remains to show that, up to relabelling, $w$ and $u^{G_0}$ are the unique transitive magic representations in $M_4(C(G_0))$. Suppose $u$ is a transitive magic representation in $M_4(C(G_0))$. From the arguments above concerning the Haar measure and the counit, we know that the first row must be $f_1+f_i$ followed by a permutation of $\{f_j+f_k,q,I_2-q\}$, for some rank one projection $q\in M_2(\mathbb{C})$. We can assume without loss of generality that
$$u=\begin{pmatrix}
      f_1+f_i & f_j+f_k & q & I_2-q \\
      * & * & * & * \\
      * & * & * & * \\
      * & * & * & *
    \end{pmatrix}.$$
Inspection of the comultiplication shows that where $A$ is the direct sum of the one dimensional summands in $C(G_0)$, and $B$ the two dimensional summand:
$$\Delta(A)\subset A\otimes A+B\otimes B\text{ and }\Delta(B)=A\otimes B+B\otimes A.$$
In addition, with the antipode on $B$ given by the transpose,  $S(u_{ij})=u_{ji}$, and with $a,b,c$ another permutation of $\{2,3,4\}$, this implies that $u$ must be of the form :
$$u=\begin{pmatrix}
      f_1+f_i & f_j+f_k & q & I_2-q \\
      f_j+f_k & f_1+f_i & I_2-q & q \\
      q^T & I_2-q^T & f_1+f_a & f_b+f_c \\
      I_2-q^T & q^T & f_b+f_c & f_1+f_a
    \end{pmatrix}.$$
(a) Suppose that $i=2$:
$$u=\begin{pmatrix}
      f_1+f_2 & f_3+f_4 & q & I_2-q \\
      f_3+f_4 & f_1+f_2 & I_2-q & q \\
      q^T & I_2-q^T & f_1+f_a & f_b+f_c \\
      I_2-q^T & q^T & f_b+f_c & f_1+f_a
    \end{pmatrix}.$$
    By looking at the $E_{11}\otimes E_{11}$ term in $\Delta(u_{11})$ and in
    $$u_{13}\otimes u_{31}+u_{14}\otimes u_{41}= q\otimes q^T+(I_2-q)\otimes (I_2-q^T),$$
     we find:
    $$q_{11}^2+(1-q_{11})^2=\frac{1}{2}\implies q_{11}=\frac{1}{2},$$
    and so
    $$q=\begin{pmatrix}
          \frac12 & \frac12 e^{i\theta} \\
          \frac{1}{2}e^{-i\theta} & \frac12
        \end{pmatrix}\text{ and }q^T=\begin{pmatrix}
          \frac12 & \frac12 e^{-i\theta} \\
          \frac{1}{2}e^{+i\theta} & \frac12
        \end{pmatrix}.$$
        Examining now the $E_{21}\otimes E_{12}$ terms in $\Delta(u_{11})$ and in $u_{13}\otimes u_{31}+u_{14}\otimes u_{41}$ we find
       $$
          \frac12 e^{-i\theta}\frac12 e^{-i\theta}+\left(-\frac12 e^{-i\theta}\right)\left(-\frac12 e^{-i\theta}\right)=\frac{i}{2}$$
          \begin{align*}
          \implies \theta&= -\pi/4
          \\ \implies q&=p,
        \end{align*}
        as in $p$ from (\ref{EQpKP}). This implies that we have:
        $$u=\begin{pmatrix}
      f_1+f_2 & f_3+f_4 & p & I_2-p \\
      f_3+f_4 & f_1+f_2 & I_2-p & p \\
      p^T & I_2-p^T & f_1+f_a & f_b+f_c \\
      I_2-p^T & p^T & f_b+f_c & f_1+f_a
    \end{pmatrix}$$
    The coefficient of $E_{12}\otimes E_{21}$ in $\Delta(u_{33})$ here is $+i/2$ but only $\Delta(f_3)$  contains such a term and we conclude that $a=3$, and $u=u^{G_0}$.

    \bigskip

    (b) Carrying out the same analysis for $i=3$ we find the rank one projection $q=p^T$. Subsequently looking at the $E_{21}\otimes E_{12}$  coefficient of $\Delta(u_{33})$ gives:
    $$u=\begin{pmatrix}
      f_1+f_3 & f_2+f_4 & p^T & I_2-p^T \\
      f_2+f_4 & f_1+f_3 & I_2-p^T & p^T \\
      p & I_2-p & f_1+f_2 & f_3+f_4 \\
      I_2-p & p & f_3+f_4 & f_1+f_2
    \end{pmatrix}$$
    Where $P$ is the permutation matrix of $(13)(24)$:
     $$u^{G_0}=PuP^{-1},$$
     and $u$ is equivalent to $u^{G_0}$

    \bigskip

    (c) If $i=4$ then we have
    $$u=\begin{pmatrix}
      f_1+f_4 & f_2+f_3 & q & I_2-q \\
      f_3+f_2 & f_1+f_4 & I_2-q & q \\
      q^T & I_2-q^T & f_1+f_a & f_b+f_c \\
      I_2-q^T & q^T & f_b+f_c & f_1+f_a
    \end{pmatrix}$$
    Examining the coefficient of  $E_{11}\otimes E_{11}$ in $\Delta(u_{11})$ we require $q=E_{11}$ or $E_{22}$. If $q=E_{11}$:
    $$u=\begin{pmatrix}
      f_1+f_4 & f_2+f_3 & E_{11} & E_{22} \\
      f_3+f_2 & f_1+f_4 & E_{22} & E_{11} \\
      E_{11} & E_{22} & f_1+f_a & f_b+f_c \\
      E_{22} & E_{11} & f_b+f_c & f_1+f_a
    \end{pmatrix}$$
    Examining $\Delta(u_{31})=\Delta(E_{11})$ we must have $a=4$ and so:
    $$u=\begin{pmatrix}
      f_1+f_4 & f_2+f_3 & E_{11} & E_{22} \\
      f_3+f_2 & f_1+f_4 & E_{22} & E_{11} \\
      E_{11} & E_{22} & f_1+f_4 & f_2+f_3 \\
      E_{22} & E_{11} & f_2+f_3 & f_1+f_4
    \end{pmatrix}=w.$$
   If we go back and let $q=E_{22}$ we get a transitive magic representation equivalent to $w$ by the permutation matrix $P$ corresponding to $(34)$.

   \bigskip

   Note finally that we cannot generate $C(G_0)$ without including $u^{G_0}$. Proposition \ref{PROPBLO} gives the rest.
\end{proof}

\subsection{Orbitals}

Where $u^{G_0}$ occurs with multiplicity $m_0\geq 1$, $r\in\{w,x,y,z\}$ appears with multiplicity $m_r\geq 0$,  and $1_{G_0}$ occurs with multiplicity $m_1\geq 0$, the action $G_0\curvearrowright X$ partitions
\begin{align*}
V&=\left(\bigsqcup_{m=1}^{m_0}V_0^{(m)}\right)\sqcup \left(\bigsqcup_{r\in \{w,x,y,z\}}\left(\bigsqcup_{m=0}^{m_r}V_r^{(m)}\right)\right)\sqcup\left(\bigsqcup_{m=0}^{m_1} V_{1}^{(m)}\right). \end{align*}
In terms of orbital types, there are twelve type $V_p\times V_p$ orbitals, and fifteen type $V_p\times V_q$  orbitals to study. Where $r,t\in\{x,y,z\}$, $r\neq t$, and $s\in\{y,z\}$, the following 15 edge relations are total:
$$V_0\times V_s,\quad V_0\times V_1,\quad V_w\times V_1,\quad  V_r^{(a)}\times V_r^{(a)},\quad V_r\times V_t, \quad V_r\times V_1,\quad V_1^{(a)}\times V_1^{(a)},\quad V_1^{(a)}\times V_1^{(b)}.$$
Here with $r\in\{x,y,z\}$, 12 orbital types remain:
$$
 V_0^{(a)}\times V_0^{(a)},\quad V_{0}^{(a)}\times V_{0}^{(b)},\quad V_0\times V_w,\quad V_0\times V_x,$$
$$V_w^{(a)}\times V_w^{(a)},\quad V_w^{(a)}\times V_w^{(b)},\quad V_w\times V_r,\quad V_r^{(a)}\times V_{r}^{(b)}.
$$
We will determine these in a series of propositions.

\bigskip

The choice to label $V$ according to Proposition \ref{ProKPB} means that blocks are consecutive integers, and to use only the six transitive magic representations listed, and not the other transitive magic representations equivalent to $u^{G_0}$ and respectively $w$, confers a fixed orientation to any four block $\{i,i+1,i+2,i+3\}$ of type $V_0$ or $V_w$. When we draw a four block as a square we assume a clockwise orientation, and when we draw a block $\{j,j+1,j+2,j+3\}$ or $\{k,k+1\}$ as a vertical line, we assume an orientation where `down' is the positive direction:
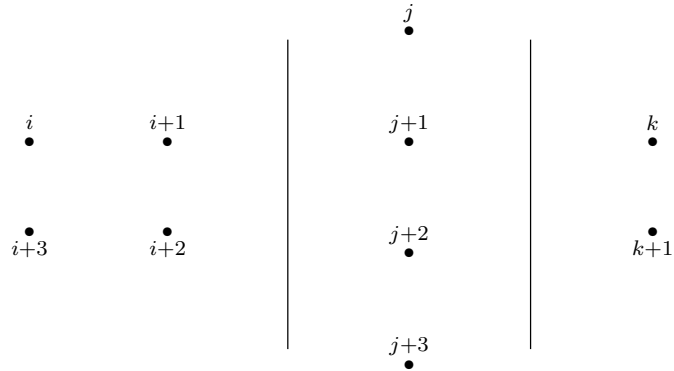
\adjustbox{scale=1,center}{%
$$\begin{tikzcd}
    &     & {} \arrow[ddd, no head] & \substack{j\\ \bullet}   & {} \arrow[ddd,no head] &     \\
\substack{i\\ \bullet}   & \substack{i+1\\\bullet} &                         & \substack{j+1\\\bullet} &                & \substack{k\\\bullet}   \\
\substack{\bullet\\ i+3} & \substack{\bullet\\ i+2} &                         & \substack{j+2 \\\bullet} &                & \substack{\bullet\\ k+1} \\
    &     & {}                      & \substack{j+3\\ \bullet} & {}             &
\end{tikzcd}$$
}

\bigskip

This first proposition implicitly describes the  $V_0^{(a)}\times V_0^{(a)}$ orbitals:
\begin{proposition}\label{ProKP1}
If $G_0\curvearrowright X$, we have that the induced subgraph $X(V_0^{(m)})$ is one of:
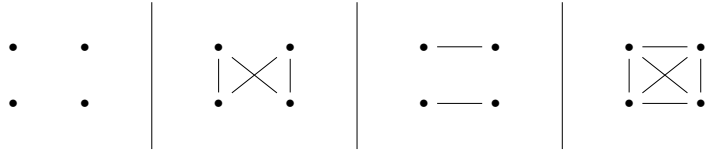
\adjustbox{scale=0.6,center}{%
\begin{tikzcd}
  &   & {} \arrow[ddd, no head] &              &              & {}                      &             &   & {}                      &                                                             &                      \\
\bullet & \bullet &                         & \bullet \arrow[rd, no head] & \bullet \arrow[ld, no head] &                         & \bullet \arrow[r, no head] & \bullet &                         & \bullet \arrow[r, no head] \arrow[d, no head] \arrow[rd, no head] & \bullet \arrow[d, no head] \\
\bullet & \bullet &                         & \bullet \arrow[u, no head]  & \bullet \arrow[u, no head]  &                         & \bullet \arrow[r, no head] & \bullet &                         & \bullet \arrow[r, no head] \arrow[ru, no head]                    & \bullet                    \\
  &   & {}                      &              &              & {} \arrow[uuu, no head] &             &   & {} \arrow[uuu, no head] &                                                             &
\end{tikzcd}
}
\end{proposition}
\begin{proof}
This follows in the same manner to that of $H_2^+\subset S_4^+$ acting on a graph (worked out as an illustration of Theorem \ref{orbitalgraph3}). Of most importance are the middle graphs which are the $V_0^{(a)}\times V_0^{(a)}$ orbitals.
\end{proof}

\begin{proposition}\label{PROKP2}
 Suppose $G_0\curvearrowright X$ with $m_0\geq 2$. Then the  $V_0^{(a)}\times V_0^{(b)}$ orbitals are:

\medskip

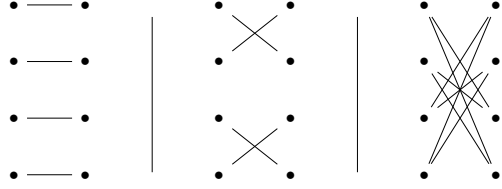
\adjustbox{scale=0.6,center}{%
$$\begin{tikzcd}
\bullet \arrow[r,no head] & \bullet & {} \arrow[ddd,no head] & \bullet \arrow[rd,no head] & \bullet & {} \arrow[ddd,no head] & \bullet \arrow[rdd,no head] \arrow[rddd,no head] & \bullet \\
\bullet \arrow[r,no head] & \bullet &                & \bullet \arrow[ru,no head] & \bullet &                & \bullet \arrow[rd,no head] \arrow[rdd,no head]   & \bullet \\
\bullet \arrow[r,no head] & \bullet &                & \bullet \arrow[rd,no head] & \bullet &                & \bullet \arrow[ruu,no head] \arrow[ru,no head]   & \bullet \\
\bullet \arrow[r,no head] & \bullet & {}             & \bullet \arrow[ru,no head] & \bullet & {}             & \bullet \arrow[ruu,no head] \arrow[ruuu,no head] & \bullet
\end{tikzcd}
$$}

\medskip

\end{proposition}
\begin{proof}
Here are the $V_0^{(a)}\times V_0^{(b)}$ orbitals.
\begin{align*}
o_d&=\{(i^{(a)}+j,i^{(b)}+j)\colon j\in\{0,\dots,3\}\}
\\ o_b&=\{(i^{(a)},i^{(b)}+1),(i^{(a)}+1,i^{(b)}),(i^{(a)}+2,i^{(b)}+3),(i^{(a)}+3,i^{(b)}+2)\}
\\ o_s&=[\{i^{(a)},i^{(a)}+1\}\times\{i^{(b)}+2,i^{(b)}+3\}]\cup [\{i^{(a)}+2,i^{(a)}+3\}\times\{i^{(b)},i^{(b)}+1\}]
\end{align*}
From Proposition \ref{ProKP1} there are two orbitals disjoint of the diagonal relation with $G_0\subset S_4^+$, and $o_b$ and $o_s$ correspond to these relations. As per Proposition \ref{ProDia}, the third orbital $o_d$ comes from the relations of the diagonal orbital.
\end{proof}
The following propositions are proven by routine calculations based on the defining formulas (\ref{KPMU})-(\ref{KPMU2}).
 \begin{proposition}\label{KP3}
    Suppose $G_0\curvearrowright X$ with $m_w\geq 1$. Then the $V_0\times V_w$ orbitals are:

\medskip

\adjustbox{scale=0.6,center}{%
$$\begin{tikzcd}
\bullet \arrow[r, no head] \arrow[rd, no head] & \bullet & {} \arrow[ddd, no head] & \bullet \arrow[rdd, no head] \arrow[rddd, no head] & \bullet \\
\bullet \arrow[r, no head] \arrow[ru, no head] & \bullet &                         & \bullet \arrow[rdd, no head] \arrow[rd, no head]   & \bullet \\
\bullet \arrow[r, no head] \arrow[rd, no head] & \bullet &                         & \bullet \arrow[ruu, no head] \arrow[ru, no head]   & \bullet \\
\bullet \arrow[r, no head] \arrow[ru, no head] & \bullet & {}                      & \bullet \arrow[ruu, no head] \arrow[ruuu, no head] & \bullet
\end{tikzcd}$$
}

\medskip

\end{proposition}

\begin{proposition}\label{PROKP4}
   Suppose $G_0\curvearrowright X$ with $m_x\geq 1$. Then the  $V_0\times V_x$ orbitals are:

\medskip

   \adjustbox{scale=0.6,center}{%
$$\begin{tikzcd}
\bullet \arrow[rd, no head] &         & {} \arrow[ddd, no head] & \bullet \arrow[rdd, no head] &         \\
\bullet \arrow[r, no head]  & \bullet &                         & \bullet \arrow[rd, no head]  & \bullet \\
\bullet \arrow[r, no head]  & \bullet &                         & \bullet \arrow[ru, no head]  & \bullet \\
\bullet \arrow[ru, no head] &         & {}                      & \bullet \arrow[ruu, no head] &
\end{tikzcd}$$
}

\medskip

\end{proposition}
\begin{proposition}\label{PROKP5}
Suppose $G_0\curvearrowright X$ with $m_w\geq 1$. Then the $V_w^{(a)}\times V_w^{(a)}$ orbitals are:

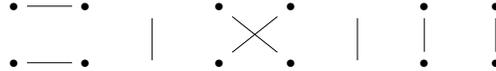
\adjustbox{scale=0.6,center}{%
$$\begin{tikzcd}
\bullet \arrow[r,no head] & \bullet & {} \arrow[d,no head] & \bullet \arrow[rd, no head] & \bullet & {} \arrow[d,no head] & \bullet \arrow[d, no head] & \bullet \arrow[d, no head] \\
\bullet \arrow[r,no head] & \bullet & {}           & \bullet \arrow[ru, no head] & \bullet & {}           & \bullet                    & \bullet
\end{tikzcd}$$
}

\end{proposition}
\begin{proof}
  By looking for $w_{ij}w_{kl}\neq 0$ we find three $V_w\times V_w$ orbitals disjoint of the diagonal:
  \begin{align*}
    o_1 & =\{(1,2),(2,1),(3,4),(4,3)\} \\
    o_2 & =\{(1,3),(2,4),(3,1),(4,2)\} \\
    o_3 & =\{(1,4),(2,3),(3,2),(4,1)\}
  \end{align*}
  \end{proof}

  \begin{proposition}\label{PROKP6}
 Suppose $G_0\curvearrowright X$ with $m_w\geq 2$. There are four $V_w^{(a)}\times V_{w}^{(b)}$ orbitals:

\medskip

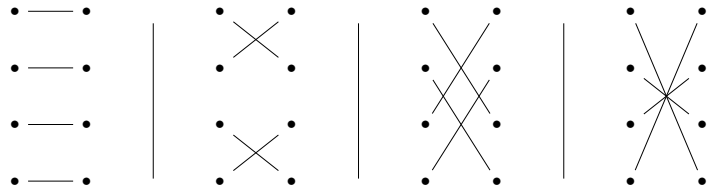
\adjustbox{scale=0.6,center}{%
 $$\begin{tikzcd}
\bullet \arrow[r, no head] & \bullet & {} \arrow[ddd, no head] & \bullet \arrow[rd, no head] & \bullet & {} \arrow[ddd, no head] & \bullet \arrow[rdd, no head] & \bullet & {} \arrow[ddd,no head] & \bullet \arrow[rddd,no head] & \bullet \\
\bullet \arrow[r, no head] & \bullet &                         & \bullet \arrow[ru, no head] & \bullet &                         & \bullet \arrow[rdd, no head] & \bullet &                & \bullet \arrow[rd,no head]   & \bullet \\
\bullet \arrow[r, no head] & \bullet &                         & \bullet \arrow[rd, no head] & \bullet &                         & \bullet \arrow[ruu, no head] & \bullet &                & \bullet \arrow[ru,no head]   & \bullet \\
\bullet \arrow[r, no head] & \bullet & {}                      & \bullet \arrow[ru, no head] & \bullet & {}                      & \bullet \arrow[ruu, no head] & \bullet & {}             & \bullet \arrow[ruuu,no head] & \bullet
\end{tikzcd}
$$
}

\medskip

 \end{proposition}
 \begin{proof}
   By Proposition \ref{ProDia}, we have three orbitals from Proposition \ref{PROKP5} as well as another coming from the diagonal relation.
 \end{proof}

 \begin{proposition}\label{PROKP7x}
   Suppose $G_0\curvearrowright X$ with $m_w,m_{x}\geq 1$. The $V_w\times V_x$ orbitals are the same as the $V_0\times V_x$ orbitals.
\end{proposition}
 \begin{proposition}\label{PROKP7y}
   Suppose $G_0\curvearrowright X$ with $m_w,m_{y}\geq 1$. There are two $V_w\times V_y$ orbitals:

\medskip

   \adjustbox{scale=0.6,center}{%
   $$\begin{tikzcd}
\bullet \arrow[rd, no head] &         & {} \arrow[ddd, no head] & \bullet \arrow[rdd, no head] &         \\
\bullet \arrow[rd, no head]          & \bullet &                         & \bullet \arrow[r, no head]            & \bullet \\
\bullet \arrow[ru, no head]          & \bullet &                         & \bullet \arrow[r, no head]            & \bullet \\
\bullet \arrow[ru, no head] &         & {}                      & \bullet \arrow[ruu, no head] &
\end{tikzcd}$$
}

\medskip

\end{proposition}
 \begin{proposition}\label{PROKP7z}
   Suppose $G_0\curvearrowright X$ with $m_w,m_{z}\geq 1$. There are two $V_w\times V_z$ orbitals:

\medskip

   \adjustbox{scale=0.6,center}{%
   $$\begin{tikzcd}
\bullet \arrow[rd, no head] &         & {} \arrow[ddd, no head] & \bullet \arrow[rdd, no head] &         \\
\bullet \arrow[rd, no head]          & \bullet &                         & \bullet \arrow[r, no head]            & \bullet \\
\bullet \arrow[r, no head]           & \bullet &                         & \bullet \arrow[ru, no head]           & \bullet \\
\bullet \arrow[ruu, no head]         &         & {}                      & \bullet \arrow[ru, no head]           &
\end{tikzcd}$$
}

\medskip

\end{proposition}
\begin{proposition}\label{PROKP8}
Let $r\in\{x,y,z\}$. Suppose $G_0\curvearrowright X$ with $m_r\geq 2$. The $V_r^{(a)}\times V_r^{(b)}$ orbitals are:

\medskip

    \adjustbox{scale=0.6,center}{%
   $$ \begin{tikzcd}
\bullet \arrow[r,no head] & \bullet & {} \arrow[d,no head] & \bullet \arrow[rd,no head] & \bullet \\
\bullet \arrow[r,no head] & \bullet & {}           & \bullet \arrow[ru,no head] & \bullet
\end{tikzcd}$$
}

\medskip

\end{proposition}

\subsection{The Frucht property}

\begin{theorem}\label{KPF}
The Kac--Paljutkin quantum group does not have the Frucht property.
\end{theorem}
\begin{proof}
Assume that $G_0\curvearrowright X$. Consider the following copy of $D_4$ in $S_4$:
\begin{equation}D_4=\{e,(12),(34),(12)(34),(13)(24),(14)(23),(1324),(1423)\}.\label{D4perm}\end{equation}
Where $s=(14)(23)$, each $\omega\in D_4$ is of the form, for some $b_i\in \{0,1\}$,
$$\omega_{(b_1,b_2,b_3)}=(34)^{b_3}(12)^{b_2}s^{b_1}.$$
Let $\omega_{(b_1,b_2,b_3)}\in D_4$ and define a permutation $\sigma_{(b_1,b_2,b_3)}$ of $V$ as follows:
\begin{align}
  V_{0}\ni i+l & \quad\longmapsto\quad i+\omega_{(b_1,b_2,b_3)}(l+1)-1.\nonumber \\
  V_{w}\ni i+l & \quad\longmapsto\quad i+[(13)(24)]^{b_1}(l+1)-1 \label{V4blo} \\
  V_{x}\ni j+l & \quad\longmapsto\quad j+(l+b_1\mod 2)\label{V2blo} \\
  V_{y}^{m}\ni j+l&\quad\longmapsto\quad j+l\\
  V_{z}\ni j+l & \quad\longmapsto\quad j+(l+b_1\mod 2)\label{Vzblo} \\
  V_{1}\ni k & \quad\longmapsto\quad k\nonumber
\end{align}
More succinctly:
\begin{enumerate}
  \item $V_0$ blocks take a full $D_4$ action.
  \item $V_w$ blocks only take a $\mathbb{Z}_2$ action.
  \item $V_x$ and $V_y$ blocks take a $\mathbb{Z}_2$ action.
  \item $V_y$ and $V_1$ blocks are fixed.
\end{enumerate}
Note in all cases that $\sigma_{(b_1,b_2,b_3)}(V_p)=V_p$.

\bigskip

We are going to show that each $\sigma_{(b_1,b_2,b_3)}$ leaves the orbitals invariant, and therefore by Theorem \ref{orbitalgraph3}, they are each automorphisms of $X$.

\bigskip

We do not have to worry about the 15 orbitals that yield total edge relations: any $V_p^{(m)}$-invariant permutation of $V$ will leave these edges invariant.

\bigskip

By Proposition \ref{PROPCheck1} we can check that $\sigma_{(b_1,b_2,b_3)}$ leaves invariant all but one orbital to show that it leaves all orbitals invariant.
\begin{enumerate}
  \item $V_0^{(a)}\times V_{0}^{(a)}$: one of two orbitals is $\Box$, which admits a $D_4$ action. We just note that the permutations given by (\ref{D4perm}) are the correct ones.
  \item $V_0^{(a)}\times V_{0}^{(b)}$: the two orbitals are given in Proposition \ref{PROKP2}. By Proposition \ref{PROPCheck1} we need only check the first two orbitals, and a quick inspection shows they are invariant under $\sigma_{(b_1,b_2,b_3)}$.
  \item $V_0\times V_w=\{i,i+1,i+2,i+3\}\times\{j,j+1,j+2,j+3\}$: for the first orbital the relation between $\{i,i+1\}$ and $\{j,j+1\}$, and the relation between $\{i+2,i+3\}$ and $\{j+2,j+3\}$ are total. Therefore permutations that leave these four subsets invariant will leave the orbital invariant. If $b_1=0$ then
  $$\{i,i+1\},\,\{i+2,i+3\},\,\{j,j+1\},\,\{j+2,j+3\}$$ are all invariant for  $\sigma_{(0,b_2,b_3)}$. If $b_1=1$, then
  $$\{i,i+1\}\longleftrightarrow \{i+2,i+3\},$$
  and, by (\ref{V4blo}), we also have
  $$\{j,j+1\}\longleftrightarrow \{j+2,j+3\},$$
  and it follows that the first orbital (and therefore both orbitals) are invariant for all $\sigma_{(b_1,b_2,b_3)}$.
  \item $V_0\times V_x=\{i,i+1,i+2,i+3\}\times\{j,j+1\}$: note that by Proposition \ref{PROKP4}, for the first orbital, the edge relations between $\{i,i+1\}$ and $\{j\}$, and between $\{i+2,i+3\}$ and $\{j+1\}$ are total. If $b_1=0$, then $\{i,i+1\}$ and $\{i+2,i+3\}$ are invariant, but $V_x$ is fixed and so the orbitals are invariant. By (\ref{V2blo}), if $b_1=1$, then $\{j\}\longleftrightarrow \{j+1\}$ switches with $\{i,i+1\}\longleftrightarrow \{i+2,i+3\}$, and the therefore the first orbital is invariant (and therefore all are invariant).
  \item $V_w^{(a)}\times V_{w}^{(a)}=\{i,i+1,i+2,i+3\}\times\{i,i+1,i+2,i+3\}$: there are only two possible actions: the $b_1=0$ identity action, that leaves everything invariant, and the $b_1=1$ action (cyclic notation):
      $$(i,i+2)(i+1,i+3),$$
      which leaves invariant the first two orbitals in Proposition \ref{PROKP5}, and hence all such orbitals.
  \item $V_w^{(a)}\times V_w^{(b)}$: similarly to (5), the action is either the identity or an involution, both leaving all orbitals invariant.
  \item $V_w\times V_x$: looking at the first orbital, and $V_w$ and $V_x$ either both fixed, or both taking an involution, we see that the orbitals are left invariant.
  \item $V_w\times V_y=\{i,i+1,i+2,i+3\}\times\{j,j+1\}$: the first of two orbitals is
  $$o_1=\{(i,j),(i+1,j+1),(i+2,j),(i+3,j+1)\}.$$
   If $b_1=0$ everything is fixed and $o_1$ invariant. If $b_1=1$ the induced permutation $\sigma_{1,b_2,b_3}'$ maps:
   $$\sigma_{1,b_2,b_3}'(o_1)=\{(i+2,j),(i+3,j+1),(i,j),(i+1,j+1)\}=o_1.$$
  \item $V_w\times V_z=\{i,i+1,i+2,i+3\}\times \{j,j+1\}$: if $b_1=0$ everything is fixed, and any orbital invariant. If $b_1=1$, $\{i,i+1\}\longleftrightarrow\{i+2,i+3\}$, and (transposition) $(j,j+1)$, preserving the first and hence all orbitals.
  \item $V_r^{(a)}\times V_r^{(b)}$: for $r\in\{x,y,z\}$ immediate as the actions on equivalent $V_r$ are the same.
\end{enumerate}

\bigskip

Finally because $u^{G_0}$ must appear in the fundamental magic representation, and because such blocks must have $D_4$ symmetry, each $\sigma_{(b_1,b_2,b_3)}$ is a distinct (classical) automorphism of $X$. However $G_{0,\text{class}}=\mathbb{Z}_2\times\mathbb{Z}_2$ and we have shown that:
$$ D_4\not\subset G_{0,\text{class}}.$$
The result follows from Theorem \ref{orbitalgraph3}.
\end{proof}
\section{Intermediate Embeddings}\label{SEC8}
In this section we consider the situation where an intermediate embedding
$$G\subsetneq G'\subset S_N^+,$$
can give an implication:
$$G\curvearrowright X\implies G'\curvearrowright X,$$
that allows us conclude that $G$ does not have the Frucht property. In this first proposition we see that the totality property of a group dual reminds of freeness.

\bigskip

We recall that if $\widehat{\Gamma}\curvearrowright X$, then $\Gamma_u=\{g_1,\dots,g_k\}$ is the set of elements of $\Gamma$ whose Fourier-type transitive magic representations appear in the fundamental magic representation.
\begin{proposition}
  If a non-abelian total discrete group dual $\widehat{\Gamma}\curvearrowright X$, then the dual of a free product
 $$\widehat{\mathbb Z_{N_1}*\ldots*\mathbb Z_{N_k}}\curvearrowright X.$$
Therefore no finite non-abelian total group dual has the Frucht property.
\end{proposition}
\begin{proof}
Suppose that $\Gamma$ is non-abelian and total $\widehat{\Gamma}\curvearrowright X$ by
$$u=\operatorname{diag}(u^{g_1},\dots,u^{g_1},\dots,u^{g_k},\dots,u^{g_k}),$$
with $g_1,\dots,g_k$ of order $N_1,\dots,N_k$. Consider the free product $\Gamma':=\mathbb{Z}_{N_1}*\ldots*\mathbb Z_{N_k}$ with respective generators $h_1,\dots,h_k$.
Consider a fundamental magic representation $u'$ of $C(\widehat{\Gamma'})$ where each instance of $u^{g_i}$ is replaced with $u^{h_i}$:
$$u'=\operatorname{diag}(u^{h_1},\dots,u^{h_1},\dots,u^{h_k},\dots,u^{h_k}).$$
Note that as $\Gamma'$ is freely generated, its dual is total, and so the orbitals between blocks of inequivalent type are full $V_p\times V_q$.  The circulant nature of the $u^{h_i}$ implies that the orbitals both within and between blocks  of equivalent type are the same as $\widehat{\Gamma}\curvearrowright X$. Hence by Theorem \ref{orbitalgraph2}, we have $\widehat{\Gamma'}\curvearrowright X$, and therefore finite $\widehat{\Gamma}$ does not have the Frucht property.
\end{proof}

This result appears to render section 6 obsolete, but in the absence of an intermediate embedding Theorem \ref{orbitalgraph3} might still be needed to show that a quantum permutation group does not have the Frucht property. Here is a partial result for the Kac--Paljutkin quantum group.

\begin{proposition}
If $G_0\curvearrowright X$ with only $u^{G_0}$, $x$, and $1_{G_0}$ appearing, then $H_2^+\curvearrowright X$.
\end{proposition}
\begin{proof}
Suppose that $G_0\curvearrowright X$ with
$$u=\operatorname{diag}(u^{G_0},\dots,u^{G_0},x,\dots,x,1_{G_0},\dots,1_{G_0}).$$
The task is to replace each of these transitive magic representations with transitive magic representations of $C(H_2^+)$, and show that this new fundamental magic representation $u'$ of $C(H_2^+)$ has the same orbitals as $u$.

\bigskip

We can replace the trivial representations $1_{G_0}$ with trivial representations $1_{H_2^+}$. Recall the fundamental magic representation of $C(H_2^+)$:
$$u^{H_2^+}=\begin{pmatrix}
              p_{11} & q_{11} & p_{12} & q_{12} \\
              q_{11} & p_{11} & q_{12} & p_{12} \\
              p_{21} & q_{21} & p_{22} & q_{22} \\
              q_{21} & p_{21} & q_{22} & p_{22}
            \end{pmatrix}.$$

 With regard to the transitive magic representation $x\in M_2(C(G_0))$, consider
$$x'=\begin{pmatrix}
    p_{11}+q_{11} & p_{12}+q_{12} \\
    p_{21}+q_{21} & p_{22}+q_{22}
  \end{pmatrix}.$$
By comparing column one and row four of $u^{H_2^+}$, we see that this is a magic unitary, and  direct calculation shows it is a representation. Consider therefore
$$u'=\operatorname{diag}(u^{H_2^+},\dots,u^{H_2^+},x',\dots,x',1_{H_2^+},\dots,1_{H_2^+}).$$

\bigskip

With the  assumption on $u$, there are only nine $G\curvearrowright X$ orbitals to check, five of which are total, and easily seen to be total for $u'$ also:
$$V_0\times V_1,\,V_x^{(a)}\times V_x^{(a)},\,V_x\times V_1,\,V_1^{(a)}\times V_1^{(a)},\,V_1^{(a)}\times V_1^{(b)}.$$
There are four orbitals to check, namely:
$$V_0^{(a)}\times V_0^{(a)},\,V_0^{(a)}\times V_{0}^{(b)},\,V_0\times V_x,\,V_x^{(a)}\times V_x^{(b)}.$$
We know from our illustration of Theorem \ref{orbitalgraph3} that the $V_0^{(a)}\times V_0^{(a)}$ orbitals of $u'$ are the same as of $u$. Furthermore by Proposition \ref{ProDia}, the $V_0^{(a)}\times V_0^{(b)}$ orbitals of $u'$ are the same as of $u$.

\bigskip

By looking at $u^{H_2^+}_{ij}x'_{kl}\neq 0$, we find that the $V_0\times V_x$ orbitals are the same as those given in Proposition \ref{PROKP4}, and finally the $V_x^{(a)}\times V_x^{(b)}$ orbitals of $u$ are also orbitals of $u'$.
\end{proof}
As well as analogues of the two dimensional $y,z\in M_2(C(G_0))$, what is missing from a full $G_0\subsetneq H_2^+\subset S_N^+$ intermediate embedding result for $G_0$ is a  transitive magic representation of $C(H_2^+)$ of the form:
$$w'=\begin{pmatrix}
       p_a & p_b & p_c & p_d \\
       p_b & p_a & p_d & p_c \\
       p_c & p_d & p_a & p_b \\
       p_d & p_c & p_b & p_a
     \end{pmatrix}.$$

Consider a type-$w$ block $V_w^{(a)}$ of $G_0\curvearrowright X$.  There are three $V_w^{(a)}\times V_w^{(a)}$ orbitals, and, by Theorem \ref{orbitalgraph2}, this gives eight possibilities for the induced subgraph $X(V_w^{(a)})$:

\medskip

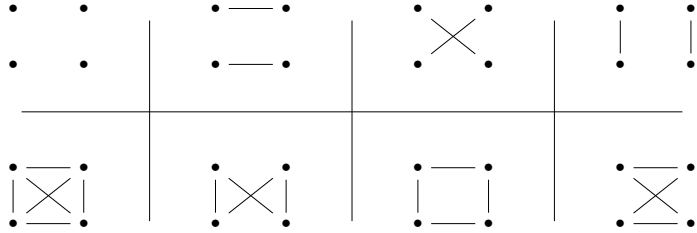
\adjustbox{scale=0.6,center}{%
$$\begin{tikzcd}
\bullet                                & \bullet           & {} \arrow[dddd,no head] & \bullet \arrow[r,no head]            & \bullet                      & {} \arrow[dddd,no head] & \bullet \arrow[rd,no head]          & \bullet           & {} \arrow[dddd,no head] & \bullet \arrow[d,no head]            & \bullet \arrow[d,no head]  \\
\bullet                                & \bullet           &                 & \bullet \arrow[r,no head]            & \bullet                      &                 & \bullet \arrow[ru,no head]          & \bullet           &                 & \bullet                      & \bullet            \\
{} \arrow[rrrrrrrrrr,no head]                  &                   &                 &                              &                              &                 &                             &                   &                 &                              & {}                 \\
\bullet \arrow[d,no head] \arrow[r,no head] \arrow[rd,no head] & \bullet \arrow[d,no head] &                 & \bullet \arrow[rd,no head] \arrow[d,no head] & \bullet \arrow[ld,no head] \arrow[d,no head] &                 & \bullet \arrow[r,no head] \arrow[d,no head] & \bullet \arrow[d,no head] &                 & \bullet \arrow[r,no head] \arrow[rd,no head] & \bullet \arrow[ld,no head] \\
\bullet \arrow[ru,no head] \arrow[r,no head]           & \bullet           & {}              & \bullet                      & \bullet                      & {}              & \bullet \arrow[r,no head]           & \bullet           & {}              & \bullet \arrow[r,no head]            & \bullet
\end{tikzcd}$$
}

\medskip

We know that the first column of graphs admits an $H_2^+$ action. We know from our illustration of Theorem \ref{orbitalgraph3} that the second column of graphs are the orbitals of $u^{H_2^+}$ and so admit an $H_2^+$ action. If we relabel the third column by $(23)$, and the fourth column by $(24)$, we see that these graphs also admit $H_2^+$ actions. Unfortunately such a relabelling does not leave invariant the $V_w\times V_x$ orbitals coming from $G_0\curvearrowright X$, and hence why we have only a partial result here. We do not believe there are analogous transitive magic representations $w',y',z'$ of $C(H_2^+)$ to complete the result in the obvious manner.

\bigskip

Through $u=\operatorname{diag}(u^{G_0},w,x)$, the following graph admits a $G_0$ action:

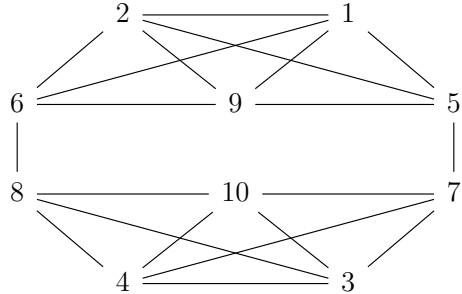
\adjustbox{scale=0.9,center}{%
$$
\begin{tikzcd}
                                                 & 2 \arrow[rr,no head] \arrow[rd,no head] \arrow[rrrd,no head] &                                      & 1 \arrow[ld,no head] \arrow[rd,no head] &                \\
6 \arrow[d,no head] \arrow[rr,no head] \arrow[ru,no head] \arrow[rrru,no head] &                                        & 9 \arrow[rr,no head]                       &                           & 5 \arrow[d,no head]  \\
8 \arrow[rr,no head] \arrow[rd,no head] \arrow[rrrd,no head]           &                                        & 10 \arrow[rr,no head] \arrow[ld,no head] \arrow[rd,no head] &                           & 7 \arrow[ld,no head] \\
                                                 & 4 \arrow[rrru,no head] \arrow[rr,no head]            &                                      & 3                       &
\end{tikzcd}
$$
}

Does $H_2^+$ act on this graph too? Or is it an explicit counterexample to $G_0\curvearrowright X\Rightarrow H_2^+\curvearrowright X$?

\bigskip

The results of this section are slightly frustrating because while we see that certain actions $G\curvearrowright X$ are not ```sufficiently quantum'' to encompass all the quantum symmetries of $X$, and we can identify a strictly larger quantum permutation group of quantum symmetries, generalising these results seems challenging. In the first instance the application of our theory requires knowing all of the transitive magic representations of $C(G)$. We leaned heavily on the results of Bichon \cite{bi3} for group duals, while the Kac--Paljutkin quantum group was sufficiently small for the classification to be straightforward. As we look at further examples, say for example the finite quantum groups of Sekine \cite{sek}, it is not clear how easy this classification question is. Indeed we do not even know if the Sekine quantum groups for parameter $k>3$ are quantum permutation groups at all.

\bigskip

We have already mentioned that the dual of the binary icosahedral group, $\widehat{2I}$, is ripe for further study. Distillation of our methods might suggest good finite dual candidates for the Frucht property. Such a candidate would give a finite graph with quantum symmetries but with finite quantum automorphism group, answering an open question of whether such a finite graph exists.

\bigskip

Finally, the non-canonical proof of Frucht of Theorem \ref{FRT} constructs for each finite group a graph:
$$G\to X_G\quad\text{ such that }\quad G(X_G)=G.$$
For small finite groups not known to be not quantum-excluding, it might be worth studying the quantum automorphism group $G^+(X_G)$ of these ``Frucht graphs''.

\end{document}